\documentclass[12pt]{article}
\usepackage{hyperref}

\usepackage{amssymb}
\usepackage{mathtools}
\usepackage{mathdots}

\usepackage{amsmath}
\usepackage{amscd}
\usepackage{amsthm}
\usepackage{graphicx}
\usepackage[all]{xy}

\usepackage{xcolor}
\usepackage{epsfig}

\usepackage{lineno}

\newcommand{\C}{\mathbb {C}}

\newcommand{\bC}{\overline{\mathbb{C}}}
\newcommand{\Arg}{\operatorname{Arg}}
\newcommand{\supp}{\operatorname{supp}}
\newcommand{\Crit}{\operatorname{Crit}}

\newtheorem{theorem}{Theorem}
\newtheorem{lemma}[theorem]{Lemma}
\newtheorem{corollary}[theorem]{Corollary}
\newtheorem{proposition}[theorem]{Proposition}      
\theoremstyle{definition}                             
\newtheorem*{definition}{Definition}

\newtheorem*{remark}{Remark}

\title{On automorphic measures, Lyapunov exponents, and instability of rational maps}
\author{Carlos Cabrera and Peter Makienko}

\begin{document}
%\linenumbers

\maketitle
\begingroup \renewcommand{\thefootnote}{} \footnotetext{This work was partially supported by PAPIIT grant IG100523. \textit{2020 Mathematics Subject Classification.} 37F10, 43A07.} \endgroup

\begin{abstract}
To construct obstructions to the stability of rational maps with non-summable critical points in their Julia sets, we introduce automorphic measures with complex eigenvalues for rational maps on the Riemann sphere. These measures extend the classical notions of quasi-invariant and conformal measures by allowing the associated Radon–Nikodym derivative to be complex-valued and proportional to a multiplicative cocycle.

We show that certain automorphic measures, particularly $(-1,1)$-uni\-modular measures, provide obstructions to stability; we also formulate related criteria in terms of associated vector fields. Specifically, if a weakly dissipative rational map admits a $( -1, 1)$-unimodular measure, then it is  $q$-unstable for some integer $q\geq 1$. This result generalizes earlier instability criteria involving pseudoconformal measures and connects the existence of such measures with the failure of structural stability.
Furthermore, we give ergodic and combinatorial conditions ensuring the existence of unimodular measures or vector fields, notably conditions based on bounded recurrence and bounded velocity of arguments. We also relate the latter to Milnor–Thurston kneading theory. These results provide a unified framework linking automorphic measures, Lyapunov spectra, and the geometric deformation spaces of rational maps.

\end{abstract}

\section{Introduction and Statements}
Our guiding theme is that recurrent critical behavior in the Julia set generates measurable cocycle structures that obstruct quasiconformal stability. To study this mechanism for rational maps with non-summable critical points in their Julia sets, we introduce a class of $\sigma$-finite complex Borel measures with complex eigenvalues, which we call automorphic measures.

Throughout the paper, $\Arg\colon\mathbb C^*\to[0,2\pi)$ denotes the  argument function. For $s,t\in\mathbb{R}$, define
\[
J_{s,t}(R)(z) =|R'(z)|^{s}\exp\bigl(-it\Arg(R'(z))\bigr), \qquad z\notin\operatorname{Crit}(R).
\]
\begin{definition} Let $R$ be a rational map, and $s,t\in\mathbb{R}$, we say that  a $\sigma$-finite complex-valued Borel measure $\rho$ is an $(s,t)$-automorphic measure with eigenvalue $\lambda\in\mathbb{C}^{*}$ if, for a measurable set $A\subset\overline{\mathbb{C}}$,
\[
\rho(A) =\frac{1}{\lambda} \int_{R^{-1}(A)}J_{s,t}(R)(z)\,d\rho(z).
\]
Equivalently, if $|\rho|$ denotes the total variation measure of the measure $\rho$, for a measurable function $\varphi\in L_1(|\rho|)$, we have:
\[
\int \varphi(R(z))J_{s,t}(R)(z)\,d\rho(z) = \lambda\int\varphi(z)\,d\rho(z).
\]
\end{definition}

\begin{remark}
Let
\[
\alpha=\frac{s-t}{2}
\qquad\text{and}\qquad
\beta=\frac{s+t}{2}.
\]
With the convention
\[
R'(z)^{\alpha}\overline{R'(z)}^{\,\beta}
:=
|R'(z)|^{\alpha+\beta}
\exp\bigl(i(\alpha-\beta)\Arg(R'(z))\bigr),
\]
we have
\[
J_{s,t}(R)(z)
   =
R'(z)^{\alpha}\overline{R'(z)}^{\,\beta}.
\]
Consequently, the automorphy relation can be written as
\[
\int
   \varphi(R(z))
   R'(z)^{\alpha}\overline{R'(z)}^{\,\beta}
   \,d\rho(z)
   =
   \lambda\int\varphi(z)\,d\rho(z).
\]
\end{remark}

In other words, an automorphic
measure represents a complex projective class of ``complex-valued quasi-invariant measures'' with
respect to $R$, whose Radon--Nikodym derivative is proportional to the function
\[j_{(s,t)}(R) = |R'|^s \left(\frac{|R'|}{R'}\right)^t\] for suitable real numbers $s$ and $t$.
If $t\in\mathbb{Z}$, then $j_{s,t}=J_{(s,t)}$ is a multiplicative algebraic
cocycle on the semigroup $\operatorname{Rat}(\overline{\mathbb{C}})$
of non-constant rational maps. More precisely, for rational maps
$Q$ and $R$,
\[
j_{(s,t)}(R\circ Q)
=
\bigl(j_{(s,t)}(R)\circ Q\bigr)j_{(s,t)}(Q)
\]
outside $\operatorname{Crit}(R\circ Q)$. Hence, for every finitely
generated semigroup of rational maps, $j_{(s,t)}$ defines a measurable
cocycle outside the critical points of its elements.

If $k\in\mathbb{Z}$, then the function
\[
j_k:=j_{(k,-k)}=(R')^{k}
\]
is an automorphy factor of exponent $k$ for subgroups of
$PSL(2,\mathbb{R})\subset \operatorname{Rat}(\overline{\mathbb{C}})$.  Real exponents
for automorphy factors are considered in
\cite{KnoppRealweights} and \cite{ManinRealweight}.

In these works, the authors define the spaces of automorphic
forms $A_k$ with real exponent $k = s = -t$ and a unitary
multiplier system $\nu$. For cyclic groups, $\nu$ is a
unitary character. Thus, every continuous functional $L$
on $A_k$ induces an example of a $(k,-k)$-current (functional)
with eigenvalue $\nu$ with respect to the action of a
Fuchsian group $\Gamma$ on the space of holomorphic functions $\phi$ on the upper
half-plane $\mathbb{H}$, defined as follows:
for $\gamma \in \Gamma$, let
\[
\gamma_k(\phi)=(\phi\circ\gamma)(\gamma')^k,
\]
then
\[
L\bigl(\gamma_k(\phi)\bigr)=\nu(\gamma)L(\phi)
\]
If $L$ is represented by a complex Borel measure on a suitable
fundamental domain $F$ for $\Gamma$, then its $\Gamma$-equivariant
extension is a $\sigma$-finite $(k,-k)$-automorphic measure with
eigenvalue $\nu$.
However, the existence of such measures, or even of an eigencurrent
for a suitable eigenvalue, does not require the existence of a
corresponding space of invariant forms.

From these observations, we adopt the term ``automorphic''
for our measures. Another use of the term ``automorphic measures'' appears in \cite{DouadyYoccozAutom}, in the setting of cyclic groups of diffeomorphisms of the
unit circle.

For rational maps, such measures were first considered in \cite{CMPVoronoi}
for $(s,t)=(-1,1)$, where they were shown to be related to the
instability of the corresponding rational map as follows:

Let $v$ be a critical value in the Julia set $J(R)$ with the lower
Lyapunov exponent
\[
\chi_{-}(v)
=
\liminf_{n\to\infty}
\frac{1}{n}\log\bigl|(R^{n})'(v)\bigr|
\geq 0.
\]
Define the following family of finite complex-valued measures:

\[
\rho_{\lambda, v} = \sum_{n \geq 0} \frac{\lambda^n \,
\delta_{R^n(v)}}{(R^n)'(v)}, \quad \text{for } |\lambda| < 1.
\]

If, as $\lambda\to 1$, the projective classes
$[\rho_{\lambda,v}]$ have a weak-$*$ accumulation point $[\rho]$
represented by a finite, nonzero, non-atomic measure $\rho$, then, as shown in
\cite{CMPVoronoi}, the map $R$ is unstable whenever the support of $\rho$
satisfies a Mergelyan-type condition. In particular, $\rho$ is an
$(-1, 1)$-automorphic measure with eigenvalue $\lambda = 1$.
We generalize this fact in Theorem \ref{th.Dos} and its corollary below.

In the summable case, the limiting measure $\rho$ is atomic. Hence,
as $\lambda\to 1$ within the unit disk, the measures
$\rho_{\lambda,v}$ converge to $\rho$ in total variation.
This situation was considered in  \cite{AstorgSum,
AvilaInfPert, CMFixedETDS, CMPVoronoi, LevinAnalytic, Makarxiv2001}
and \cite{MakRuelle}, where the instability of a rational map
is established whenever the support of $\rho$ again
satisfies  the Mergelyan-type condition mentioned above.
Additionally, the non-Mergelyan-type condition and negative
lower Lyapunov exponent are discussed in \cite{AstorgCollet, CMFixedETDS}
and \cite{CMPVoronoi},  respectively. The instability criteria developed in the works
mentioned above are related to the problem of detecting transverse directions
to quasiconformal deformations, which may arise from the asymptotic behavior
of derivatives along critical orbits. In particular, summability methods such as
Abel and Voronoi procedures were used to extract nontrivial limiting objects from the series
\[
\sum_{n\geq 0}\frac{\lambda^n}{(R^n)'(v)}.
\]
In the summable case, these procedures lead naturally to atomic objects and instability criteria.

In the non-summable situation, however, the limiting behavior becomes considerably more subtle.
Instead of producing convergent sums, one is naturally led to consider weak limits,
invariant currents, and eigenmeasures associated with weighted transfer operators.
In this spirit, we combine  Corollary 4.11 in \cite{CMFixedETDS} and
Corollary 3.5 in \cite{CMPVoronoi} into the following proposition.

\begin{proposition}\label{pr.distribution}
Let $R$ be a rational map admitting a non-trivial quasiconformal deformation. Then
for every critical value $v$ with $\chi_-(v) = 0$,  the following statements hold:

\begin{enumerate}

 \item As $\lambda\to 1$, the projective classes
$[\rho_{\lambda,v}]$ have an accumulation point represented by a
nonzero $R$-invariant current (distribution) $D_v$ on the space of continuous functions with distributional $\overline{\partial}$
 derivatives in $L_\infty(\C)$.

 \item If the current $D_v$ is unique, then $R$ admits a
 finite $R$-invariant measure absolutely continuous with
 respect to  Lebesgue measure.

 \item If $D_v$ is represented by a measure $\rho$,
 then $\rho$ is a $(-1,1)$-automorphic measure with
 eigenvalue equal to $1$.

\end{enumerate}
\end{proposition}
 From this perspective, automorphic measures arise as a natural extension
of the summability framework. Rather than encoding the dynamics through
the convergence properties of weighted series alone, they capture the
asymptotic behavior of the derivative cocycle via Radon--Nikodym relations.
Thus, they extend the instability mechanisms previously obtained through
summability methods.

In addition,  automorphic measures naturally  generalize the notion of
quasi-invariant measures for non-injective, non-singular
endomorphisms as well as the notion of conformal measures. Indeed, let $e:X\to X$ be a measurable endomorphism of a standard measure
space $(X,|\rho|)$, where $|\rho|$ denotes the total variation measure
of a complex-valued $\sigma$-finite Borel measure $\rho$. We call $e$
\textit{backward non-singular with respect to $\rho$} if and only if $e$ is non-singular with respect to $|\rho|$, that is,
\[
|\rho|(A) = 0 \Leftrightarrow |\rho|(e^{-1}(A)) = 0.
\]
We call $e$  \textit{forward non-singular with respect to $\rho$} when
\[
|\rho|(A) =0 \Leftrightarrow |\rho|(e(A)) = 0,
\]
for every measurable set $A$.
If $e$ is bijective, then the two notions coincide. We now extend the notion of a quasi-invariant measure to the complex-valued setting (cf. \cite{BezJorgensen}).

\begin{definition}
We call a complex-valued $\sigma$-finite measure $\rho$ on $X$
\emph{quasi-invariant with respect to $e$} if the pushforward measure
\[
e_{*}(\rho)(A):=\rho(e^{-1}(A)) \qquad \textnormal{or} \qquad e_*(\rho)=\rho \circ e^{-1}
\]
is equivalent to $\rho$.  Let
\[
\phi=\frac{d\rho}{d(e_{*}(\rho))}
\]
be the Radon--Nikodym derivative. Then, setting $\omega=\phi\circ e$, for every
$g\in L^{1}(X,|\rho|)$, we have
\[\int g(e)\, d\rho=\int \frac{g}{\phi} \cdot  \,d\rho \qquad \textnormal{ or} \qquad
\int g(e) \, \omega \, d\rho = \int g \, d\rho.
\]
\end{definition}
Assume, in addition, that the restriction of $R$ to
$\operatorname{supp}(\rho)$ is injective. Then $\rho$ is
quasi-invariant with respect to $R$, and
\[
\omega_{(s,t)}
:=
\frac{1}{\lambda}J_{s,t}(R)
=
\left(\frac{d\rho}{d(\rho\circ R^{-1})}\right)\circ R
=
\frac{d(\rho\circ R)}{d\rho},
\qquad \rho\text{-a.e.}
\]
This derivative is the complex analogue of the
\textit{Radon--Nikodym derivative (multiplicative cocycle)} for the
measurable cyclic group action generated by the restriction of $R$ to
the measure space
\[
\bigl(\operatorname{supp}(\rho),|\rho|\bigr)
\]
(cf.~\cite{BezJorgensen,SchmidtRig}). Equivalently, $\rho$ may be regarded as a complex
$\sigma$-finite conformal measure for this cyclic group action.
However, $\rho$ need not be a ``complex conformal measure'' for the
rational map $R$ itself, because $R$ need not be forward nonsingular
with respect to $\rho$ when
\[
\operatorname{supp}(\rho)\neq J(R)
\qquad\text{and}\qquad
\deg(R)\geq 2.
\]

Krieger \cite{Kriegerquasi} proved that every continuous uniquely ergodic automorphism of a compact space admits an $\exp(f+k)$-conformal probability measure for each continuous real-valued function $f$ and a suitable constant $k$.

For Kleinian groups, the Patterson--Sullivan conformal measures
with exponent $s$ are examples of $(s,0)$-automorphic measures
with eigenvalue $\lambda = 1 $.

Note that an $s$-conformal measure $m$ for a rational map $R$ is an
$(s,0)$-automorphic measure with eigenvalue $\lambda=\deg(R)$.
Indeed, according to \cite{BezJorgensen,DenkerUrbanski,MakarovSmirnovI},
an $s$-conformal measure $m$ satisfies
\[
m(R(A))
=
\int_A \omega_m\,dm,
\qquad
\omega_m=|R'|^s,
\]
for every measurable set $A$ on which the restriction $R|_A$ is
injective.

Let
\[
\mathcal{L}_R(\phi)(x) = \sum_{R(y) = x} \frac{\phi(y)}{\omega_m(y)}
\]
be the \textit{Perron--Frobenius operator} associated with
$\omega_m$ and $R$ with respect to the measure $m$.
Then $\mathcal{L}_R$ is a bounded positive operator on $L^{1}(m)$
 satisfying
\[
\int \mathcal{L}_R(\phi)\,dm
=
\int \phi\,dm
\]
for every $\phi\in L^{1}(m)$; see
\cite{DenkerUrbanski,MakarovSmirnovI}. Equivalently, a measure $\nu$ is conformal with conformal factor
$\omega_m$ if and only if the associated measure functional
\[
\ell_\nu(\phi):=\int \phi\,d\nu
\]
is $\mathcal{L}_R$-invariant, that is,
\[
\ell_\nu\bigl(\mathcal{L}_R(\phi)\bigr)=\ell_\nu(\phi)
\]
for every $\phi$ for which the integrals are defined.
On the other hand, define
\[
T_R(\phi)(x):=\phi(R(x))\,\omega_m(x).
\]
Then, for every $x$ outside the set of critical values of $R$,
\[
\mathcal{L}_R\bigl(T_R(\phi)\bigr)(x)
=
\deg(R)\,\phi(x).
\]
Consequently, if the set of critical values has zero $m$-measure,
then
\[
\mathcal{L}_R\circ T_R
=
\deg(R)\,\operatorname{Id}
\]
as operators on $L^{1}(m)$.
Therefore, every $\mathcal{L}_R$-invariant functional is a
$T_R$-eigenfunctional with eigenvalue $\deg(R)$.

Finally, $(0,0)$-automorphic measures with $\lambda = 1$ correspond exactly to
$R$-invariant $\sigma$-finite measures.

\begin{definition}
We call a finite  $(s,t)$-automorphic measure with eigenvalue
$|\lambda| = 1 $ an
\textit{$(s,t)$-unimodular measure}.
\end{definition}
 
The simplest examples of unimodular measures are suitably weighted
finite linear combinations of Dirac measures supported on indifferent
periodic orbits. In Lemma~\ref{lem.simptransversals} below, we show that Herman rings
and Siegel disks admit non-atomic $(-1,1)$-unimodular measures.

\begin{definition}
We call a probability $(s,0)$-automorphic measure
with eigenvalue $\lambda = 1$ an $s$-pseudoconformal measure.
\end{definition}
If a nonzero $(s,t)$-unimodular measure $\rho$ satisfies
\[
d\left((|R|^s|\rho|)\circ R^{-1}\right)=d|\rho|,
\]
then the normalized total variation measure
\[
\frac{|\rho|}{|\rho|(\overline{\mathbb{C}})}
\]
is an $s$-pseudoconformal measure.

We give special attention to $(-1,1)$-unimodular measures. We also
consider their total variation measures whenever these are scalar
multiples of $(-1)$-pseudoconformal measures. Nevertheless, some
of our arguments apply more generally to $(s,t)$-unimodular and
$(s,t)$-automorphic measures.

\begin{definition}
 Let $PM(d)\subset Rat_d$ be the closure of rational maps  admiting a $-1$-pseudoconformal measure.
\end{definition}

The following results relate.

\begin{theorem}\label{th.Uno}
Let $R$ be a weakly dissipative rational map then $R\notin PM(d)$ if and only if $R$ is hyperbolic. In particular, a $J$-nonstable map with no summable critical points admits a $-1$-pseudoconformal measure.
\end{theorem}

The following corollary is an immediate consequence of Theorem~\ref{th.Uno}
and Theorem~D of \cite{MSS}.

\begin{corollary}\label{cor.Rhyper}
 If a structurally stable rational map $R$ admits an $s$-pseudoconformal
 measure with $s < 0$, then  $J(R)$ supports an invariant line field.
\end{corollary}

Motivated by the Fatou conjecture, we formulate the following.

\medskip

\textbf{Instability conjecture}: \textit{A rational map $R$ admitting a
\emph{pseudoconformal} measure with negative exponent is  unstable.}
\medskip

We confirm this conjecture under slightly stronger assumptions. To
state the result precisely, we need additional preparation.

Let $Rat_d$ be the space of rational maps of
degree $d$. For $R \in Rat_d$, let $H(R) \subset Rat_d$ be the set of all rational
maps which are Hurwitz-equivalent to $R$, and let $RH(R) \subset H(R)$ be the
reduced Hurwitz class of $R$ (see definitions below in the preliminaries).

For a generic rational map $R$, the classes $H(R)$ and $RH(R)$ are
complex orbifolds of complex dimension $2d + 1$ and $2d - 2$, respectively.
Otherwise, $H(R)$ is an analytic subset in $Rat_d$ and, hence, has (locally) positive codimension.
Therefore, $H(R^q)\subset Rat_{d^q}$ has positive codimension for
every $q \geq 2$.

Let $QC(R)$ be the set of rational maps quasiconformally conjugated to $R$. We say that
$R \in M \subset Rat_d$ is \textit{structurally stable in the submanifold}
$M$ if $QC(R) \cap M$ is an open  subset of $M$.
For a natural number $q$, we say that the map $R$ is \textit{$q$-stable} if
$R^q$ is structurally stable in the family $RH(R^q) \subset Rat_{d^q}$. Note
that $\operatorname{dim}(QC(R^q))$ can be greater than $0$ even when $\operatorname{dim}(QC(R)) = 0$. For
example, the map $f_{1/4}(z) = z^2 + 1/4$ is quasiconformally rigid in the
quadratic family, but its second iteration admits a non-trivial quasiconformal
deformation within polynomials of degree $4$.

The following theorem establishes how stability properties propagate from $q$-stability to
all divisors of $q$.

\begin{theorem}\label{th.qstable}
For a structurally stable map $R$, if $R$ is $q$-stable, then it is $k$-stable for all $k|q$.
If $R$ is hyperbolic, then it is $q$-stable for $q \ge 2$.
\end{theorem}

Now, let $C(R)$ denote the set of critical points of the rational map  $R$,  $V(R) = R(C(R))$
be the \textit{set of critical values}, and
$\operatorname{Crit}(R) = \bigcup_{n \geq 1} C(R^n)$ be the backward orbit
of $C(R)$. For $c \in C(R)$, let
\[
P_c(R) = \overline{\bigcup_{n > 0} R^n(c)}
\]
be the \textit{individual
postcritical set}, and
\[
 P(R) = \overline{\bigcup_{c \in C(R)} P_c(R)}
\]
be the \textit{postcritical set} of $R$.

A measurable set $A$ is
\textit{Lebesgue weakly wandering} with respect to $R$
if the backward orbit $\{R^{-n}(A)\}$
contains infinitely many pairwise disjoint sets modulo Lebesgue measure.
The union, modulo Lebesgue measure, of all such sets, denoted $WD(R)$, is called the
\textit{weakly dissipative set}  with respect to $R$, and its complement, modulo Lebesgue measure,
\[
 SC(R) = \overline{\C} \setminus WD(R)
\]
is called the \textit{strongly conservative set} with respect to $R$.
By the classical theorem of Kakutani, the Lebesgue measure of $SC(R)$
is positive if and only if $SC(R)$ carries a probability invariant measure absolutely continuous with respect to the Lebesgue measure (see, for example, \cite{Krengel}).

We will say that $R$ is \textit{weakly dissipative} whenever
the Lebesgue measure of $SC(R) \cap P(R)$ is zero. For a generic rational
map $R$,  $SC(R) \subset P(R)$ modulo Lebesgue measure (see, for instance,
Proposition 4.1 of \cite{CMFixedETDS}, or Lemma \ref{lem.SC} below). Conjecturally, $P(R)$ is either the
whole sphere or has zero Lebesgue measure; in other words,
conjecturally a generic  rational map is weakly dissipative.

We are now ready to formulate one of the main theorems of the paper. It generalizes the principal results of  \cite{CMPVoronoi}, relates $q$-stability to the hyperbolic setting, and gives a central instability criterion in terms of $(-1,1)$-unimodular measures.

\begin{theorem}\label{th.Dos}
Let $R$ be a weakly dissipative rational map which admits a $(-1, 1)$-unimodular
measure $\rho$. Then there exists an integer $q \geq 1$ such that $R$ is
$q$-unstable.
\end{theorem}

We obtain the following immediate corollary.

\begin{corollary}\label{cor.sigma}
Under the assumptions of Theorem~\ref{th.Dos},  $R$ is structurally unstable whenever the eigenvalue associated with $\rho$  is $\lambda=1$.
\end{corollary}

\begin{remark} \textit{Eigenmeasures and transversality}.

In the weakly dissipative case, by Corollary~\ref{cor.sigma}, a unimodular measure $\mu$ determines a nontrivial transverse direction to the infinitesimal quasiconformal deformations in the tangent space $T_R(\operatorname{Rat}_d)$ by defining a linear functional on quasiconformal vector fields of the form
\[
v \mapsto \int v \, d\mu.
\]

This functional vanishes on infinitesimally trivial deformations, namely on vector fields of the form
\[
\frac{v}{R'} = \frac{X \circ R}{R'} - X,
\]
and therefore descends to a nonzero element in the dual of the quotient of quasiconformal vector fields by such coboundaries. In particular, $\mu$ detects a transverse direction to the orbit of $R$
under quasiconformal conjugacies, providing an analog of Lyubich's notion of
transversality to hybrid classes of quadratic polynomials, \cite{LyuI}.

When the eigenvalue satisfies $\lambda \neq 1$, the same conclusion holds
for a suitable iterate $R^q$.

In the absence of weak dissipativity, this mechanism may degenerate. One case to keep in mind is that in which the postcritical set  or an individual postcritical set of $R$ coincides with the whole Riemann sphere.
Here, detecting  instability may require an additional geometric or dynamical structure.

\end{remark}

\subsection{Unimodular vector fields}

We begin this subsection with a definition.
\begin{definition}
Let $\nu$ be a $(-1,1)$-unimodular measure with eigenvalue $\lambda$. If the total variation measure $|\nu|$ is a $(-1)$-pseudoconformal measure then, the Radon--Nikodym derivative
$v=d\nu/d|\nu|$ defines a measurable vector field satisfying
\[
v(R)\frac{R'}{|R'|} = \overline{\lambda} \, v, \quad  w \textnormal{- a.e.}
\]
In this situation, we say that $v$ is a \textit{unimodular vector field with
eigenvalue $\lambda$}.
\end{definition}

Thus, a unimodular vector field that is nonsingular with respect to a
$(-1)$-pseudoconformal measure generates a $(-1,1)$-unimodular measure.

We next consider  ergodic and recurrence conditions, together with combinatorial
constraints that guarantee the existence of unimodular measures or vector
fields. This leads to concrete instability criteria for specific dynamical
configurations.

Observe that if $\lambda$ is $-1$ or $1$, then $\mu = (\overline{v})^2$ defines
a measurable fixed point of the Beltrami operator, often singular with respect to
Lebesgue measure.  From Corollary~\ref{cor.Rhyper}, if $R$ is
structurally stable, then a pseudoconformal measure induces an invariant line
field. Hence, we obtain the following elementary conclusion reciprocal to the
observation above.

\begin{proposition}\label{pr.principal}
If $\mu$ is an invariant line field for a rational map $R$ which is nonsingular
with respect to a $(-1)$-pseudoconformal measure $\rho$, then $R$ admits
a $(-1,1)$-unimodular measure with eigenvalue $\lambda \in \{-1, 1\}$.
\end{proposition}

Note that if $\rho$ is a pseudoconformal measure as in the previous proposition
and $\mu$  is a nonzero invariant line field, continuous on $J(R) \cap U$ for
a suitable open set $U \subset \overline{\C}$, then $\mu$ is non-singular
with respect to $\rho$.
In particular, this is certainly true when $\mu$ is
locally represented by a restriction to the Julia set of a holomorphic
pullback of the standard line field
$\frac{\partial \overline{z}}{\partial z}$.

In Section~\ref{proofs}, we will show that every foliated Fatou component
admits a unimodular $(s, t)$-measure for every pair of real numbers
$s$ and $t$.

As shown in \cite{CMPVoronoi} (see also the discussion in
Section~\ref{proofs}), there exists a class of rational maps for which pseudoconformal measures are unimodular measures with eigenvalue $\lambda = 1$. For
example, if $R' \geq 0$ on the support of an $s$-pseudoconformal measure $
\rho$, then $\rho$ is an $(s,-s)$-unimodular measure with eigenvalue $
\lambda = 1$. If $R' \leq 0$ on the support of $\rho$, then $\rho$ is an $
(s,-s)$-unimodular measure with eigenvalue $\lambda = -1$. We extend this
observation to a broader class of holomorphic maps defined as follows.

\begin{definition}
We say that a meromorphic function $R \colon D \to \overline{\C}$ on a domain
$D$ is a $\gamma$-endomorphism with respect to a compact set
$C \subset \gamma$ if $\gamma$ is an oriented differentiable arc,
$C$ is $R$-invariant, and the restriction of $R$ to $C$  either
preserves or reverses the orientation of $\gamma$.
\end{definition}

Simple examples of holomorphic $\gamma$-endomorphisms are rational circle endomorphisms, defined as follows.

\begin{definition}
A rational map $R$ is a \emph{rational circle endomorphism} if there exists an invariant circle $C$ such that the restriction $R \colon C \to C$ is a covering map of degree $k \leq \deg(R)$.
\end{definition}

Note that a rational circle endomorphism may have critical points on $C$. Other examples of $\gamma$-endomorphisms are given by real meromorphic functions with either nonpositive or nonnegative derivatives on the real line $\mathbb{R}$. Now, let $R$ be a holomorphic $\gamma$-endomorphism with respect to a compact set $C \subset \gamma$, and let $o$ be an orientation on the arc $\gamma$. Then $o$ induces a vector field $v_o$ consisting of unit tangent vectors to $\gamma$ that are positively oriented with respect to $o$. If $R$ preserves $o$, then
\[
v_o(R)\,\frac{R'}{\lvert R' \rvert} = v_o
\]
off the critical set, while
\[
v_o(R)\,\frac{R'}{\lvert R' \rvert} = -v_o
\]
whenever \(R\) reverses the orientation determined by \(o\). We conclude this discussion with the following proposition.

\begin{proposition}\label{pr.circleendo}
Let $R$ be a holomorphic $\gamma$-endomorphism with invariant compact set $C \subset \gamma$, and let $c \in C$ be a critical point with $\chi_{-}(c) = 0$. Then $R$ admits a $(-1,1)$-unimodular measure with eigenvalue $\lambda \in \{-1,1\}$.
\end{proposition}

For $R\in Rat(\overline{\C})$ and $t\in \mathbb{Z}$, the map
$$z \mapsto \left( \frac{R'(z)}{|R'(z)|} \right)^t$$ is a measurable
map from  $\overline{\C}$ to the unit circle $\mathbb{S}^1$. Hence,
for each integer $t$, the endomorphism $R_t$ of
$\overline{\C} \times \mathbb{S}^1$ given by
$$R_t(z,v) = \left(R(z), \left( \frac{R'(z)}{|R'(z)|}
\right)^t \cdot v\right)$$
is a measurable skew-product associated with the multiplicative
cocycle
\[
 A_t(n,z) = \left[ \frac{(R^n)'(z)}{|(R^n)'(z)|} \right]^t, \quad  n \geq 0.
\]
The endomorphism $R_1$ may be regarded as the action of the rational
map $R$ on the unit tangent bundle $UT(\overline{\C})$ away from the
critical points. Indeed, if  $S_2 = \overline{\C} \setminus V(R)$
and $S_1 = R^{-1}(S_2)$, then $$R_1 \colon UT(S_1) \to UT(S_2)$$
is surjective. Therefore, the following proposition can be treated as a
generalization of Proposition~\ref{pr.principal} to
non-real eigenvalues. For its formulation, we need the following definition:

Let $(X, \rho)$ and $(Y, m)$ be measure spaces. A measurable function $f$
on $(X \times Y, \rho \times m)$ is called \textit{constant with respect to the
first variable} if, for $m$-almost every $y_0 \in Y$, the function
$g(x) = f(x, y_0)$ is constant for $\rho$-almost all $x$.

\begin{definition}
Let $(X, \rho)$ and $(Y, m)$ be measure spaces.  We say that an endomorphism
$e$ of a measure space $(X \times Y, \rho \times m)$ is \textit{weakly mixing with
respect to the first variable} if the only solutions of the equation

 \begin{equation}\label{eqn:1}
  \phi(e) = \lambda \phi
  \end{equation}
for complex $\lambda$ with $\lvert \lambda \rvert = 1$ are functions
constant with respect to the first variable.
\end{definition}

\begin{proposition}\label{pr.pseudoconformalerg}
Let $\rho$ be a $(-1)$-pseudoconformal measure for a rational map $R$.
Then there exists a unimodular vector field $v$ with a complex eigenvalue $|\lambda| =1$
if and only if the restriction of $R_1$ to
$\operatorname{supp}(\rho) \times \mathbb{S}^1$ is not weakly mixing
with respect to the first variable.
\end{proposition}
 
\textbf{Remark}. If $R$ is injective on the support of a
$(-1)$-pseudocon\-formal measure $\rho$, then the equation~\eqref{eqn:1}
with respect to $R_1$ on the dissipative set
$D(R_1) \subset \operatorname{supp}(\rho) \times \mathbb{S}^1$
has a measurable solution bounded $(\rho \times m)$-almost everywhere
for every complex $\lambda$ with $\lvert \lambda \rvert = 1$, whenever
$(\rho\times m)(D(R_1)) > 0$.

Let $\rho$ be an invariant probability measure for $R$, and
\[
R_1: (\operatorname{supp}(\rho) \times \mathbb{S}^1, \rho \times m)\to
(\operatorname{supp}(\rho) \times \mathbb{S}^1, \rho \times m)
\]
be the restriction of $R_1$ to
$X = \operatorname{supp}(\rho)\times \mathbb{S}^1$.

Let $Y = \operatorname{supp}(\rho) \times \mathbb{R}$
be a covering space over  $X$ with projection $p(x,t)=(x,e^{it})$.
Let
\[
r_0 =\int_{\overline{\C}} \operatorname{Arg}(R') \, d\rho,
\quad 0 \leq \operatorname{Arg}(z) < 2\pi.
\]
Then
\[
 \tilde{R}(x,t) = (R(x), t + \operatorname{Arg}(R'(x))- r_0)
\]
is a skew-product with respect to the real cocycle
$\phi(z) = \operatorname{Arg}(R'(z)) - r_0$.

In fact, $\tilde{R}$ is the translation of the lift of $R_1$
to $Y$ given by $(x,t) \mapsto (x, t + r_0)$.
Let $\widetilde m$ be the pullback of $m$ under $p$. The following proposition is a simple application of ergodic theory.

\begin{proposition}\label{pr.existunimodular}
Assume that $\tilde{R}$ admits on $Y$ a probability invariant measure $\nu$ equivalent to $\rho \times \widetilde{m}$. Then there
exists a unimodular vector field with eigenvalue $\lambda = e^{i r_0}$.
\end{proposition}

\begin{definition}
We say that a point $x\in \overline{\C}$ is \emph{boundedly recurrent} with respect to a rational map $R$ if there exists a constant $M > 0$ such that for every $\epsilon > 0$ there exists $k \in \mathbb{N}$ with the following properties:
\begin{itemize}
  \item The set
  \[
  O_k(x) = \bigcup_{n \geq 0} (R^{k})^n(x)
  \]
  has diameter at most $\epsilon$.
  \item For every $0 \leq i \leq k-1$,
  \[
  \operatorname{diam}(R^i(O_k(x))) \leq M \epsilon.
  \]
\end{itemize}
\end{definition}

We can now state one of the main theorems:

\begin{theorem}\label{th.unimodular}
Let $R$ be a rational map. If $c$ is a boundedly recurrent critical point with $\chi_-(R(c)) = 0$ and
\[
 \liminf_{n \to \infty} \frac{1}{n} \sum_{i=0}^{n-1} \frac{1}{\lvert (R^i)'(R(x)) \rvert} > 0 \tag{*}\label{eq.unimodular}
\]
holds for every $x\in P_c(R)$, except possibly those in a countable set. Then, $R$ admits a non-zero unimodular measure.
\end{theorem}

The following corollary can be regarded
as a virtual version of the instability
conjecture for weakly dissipative
rational maps.

\begin{corollary}\label{cor.unimodularexis}
Let $R$ and $c \in J(R)$ be as in Theorem~\ref{th.unimodular}. Then there
exists  $q \geq 1$ such that $R$ is not $q$-stable whenever $R$ is weakly
dissipative.
\end{corollary}

The next theorem gives conditions under which $(-1,1)$-automorphic measures arise
under fairly general assumptions.

Let $f_r(z) = \Arg(R'(z)) - r$, for $r \in \mathbb{R}$, be a family of bounded measurable cocycles.

\begin{theorem}\label{th.conservative}
Assume a rational map $R$ admits a  $(-1)$-pseudoconformal measure $\rho$. If $g=f_{r_0}$ is a conservative cocycle with respect to $\rho$ for some suitable $r_0$, then there exists a $(-1,1)$-automorphic measure $m$ with eigenvalue $\lambda=e^{ir_0}$.

\end{theorem}

The number $r_0$ may be computed as it is shown in the  next proposition.

\begin{proposition}\label{pr.pseudoconformal}
Assume that a rational map $R$ admits a  $(-1, 1)$-unimodular
measure $\rho$ with eigenvalue $\lambda = e^{ir_0}$, which is equivalent to an invariant
probability measure $\rho$. Then
\[
\int \Arg(R') \, d\rho = r_0,
\]
whenever $0 \leq \Arg(R') < 2\pi$, and

\[
\int  \ln |R'| \, d\rho = 0,
\]
whenever $\ln |R'|$ is $\rho$-integrable.
\end{proposition}

Next, we give combinatorial conditions that provide a class
of rational maps that admit unimodular vector fields with $s = -1$ and
$t = 1$ without condition~\eqref{eq.unimodular} of
Theorem \ref{th.unimodular}. We observe that this class is
non-empty and, in particular, contains the Feigenbaum polynomial.

Let $\{c_n\}$ be a sequence of nonzero complex numbers, and set
$b_n = \prod_{k=0}^{n} c_k$. Define
\[
K_n = \frac{1}{2\pi} \left[ \Arg(b_{n}) + \Arg(c_{n+1}) - \Arg(b_{n+1}) \right].
\]
Then $K_n \in \{0,1\}$ is a two-valued function, and hence
\[
\sum_{i=0}^{n} \Arg(c_i) = \Arg(b_n) + 2\pi \sum_{i=0}^{n-1} K_i \textnormal{ with } \sum_{i=0}^{n-1} K_i\leq n-1
\]

\medskip

\textbf{Remark.} The sequence $\{K_n\}$ defines a two-valued, nearly additive cocycle, which can be regarded as a “kneading” of the arguments of $c_n$. In particular, for $R(z) = z^2 + v$ with $v \in \mathbb{R}$ and $c_n = R'(R^n(v))$, the sequence $\{K_n\}$ can be recovered from the classical Milnor--Thurston kneading sequence of the critical value $v$.

\begin{definition}[\textit{Bounded velocity}] A point $v$ has \emph{bounded velocity of arguments} with respect to $R$ if the sequence $\{K_n\}$ constructed for $\{c_n=R'(R^n(v))\}$ satisfies
\[
\sum_{i=1}^n K_i = \rho n + O(1), \quad \, 0\leq \rho \leq 1.
\]

\end{definition}

Let $\Omega_R(z)$ be the omega-limit set of  $z$ under $R$.
\begin{definition}[Joint Continuity set]
Let $X$ be an $R$-invariant set. Let $JC(X)$ be the set of  points $x\in X$ for which there is a branch of logarithm such that $\ln (R^n)'$ is continuous at $x$ for every $n>0$. The set $JC(X)$ is the set of joint continuity.

\end{definition}

\begin{theorem}\label{th.boundedvel}
 Let $R$ be a rational map and  $v \in \C$ be a point of bounded velocity of
 arguments with respect to $R$. Then $R$ admits a $(-1, 1)$-unimodular vector
 field on the set $JC(\Omega_R(v)) $ with eigenvalue $\lambda = \exp(2\pi i\rho) $.
\end{theorem}

\begin{remark}
\leavevmode\par
\begin{itemize}
 \item  Let $P(z)=z^2 + v$ be the Feigenbaum polynomial. Then the
sequence
\[
\{\operatorname{sign}(P'(P^n(v)))\}
= \{\operatorname{sign}(P^n(v))\}
\]
is the doubling sequence over the alphabet $\{-1,1\}$. Thus, we conclude that $v$ has bounded velocity of arguments.
\item Assume that $X$ is an $R$-invariant set and there exists a Jordan arc
$\gamma$ connecting $0$ to infinity so that $R'(\overline{X})\subset \overline{\C}\setminus \{\gamma\}$, then the set of joint continuity $JC(X)=X$. If $R'(\overline{X})\subset (\overline{\C}\setminus \{\gamma\})\cup \{0\}$, then $JC(X)=X\setminus \operatorname{Crit}(R)$.
Even more, if $X=\Omega_R(z)$ does not contain critical points, then by the celebrated theorem of Gottschalk-Hedlund \cite[Chapter 14]{GottschalkHedlund}, the following statements are equivalent:

\begin{enumerate}
 \item There exists a continuous function $h$ such that the equation
 \[
 \Arg(R')-\lambda=h-h(R)
\]
holds on $X$.
\item There exists $x\in X$ with bounded velocity.
\item Every point of $X$ has bounded velocity.

\end{enumerate}

If $\Omega_R(z)$ contains a critical point $c$ and is totally disconnected, then \[
JC(\Omega_R(z))=\Omega_R(z)\setminus \Crit(R).                                                                               \]
                                                                                  For instance, this holds when $c$ is an infinitely renormalizable critical point with \textit{a-priori} bounds.

Furthermore, for every $x$ and every branch of $\Arg(R')$,  \[JC(\Omega_R(x))=\Omega_R(x)\setminus \Crit(R),\] whenever $R'$ takes real values on $\Omega_R(x)$.
\end{itemize}
\end{remark}

The previous remark and theorems lead to the following corollary, where bounded recurrence combined with Lyapunov conditions forces the existence of a unimodular measure.

\begin{corollary}\label{cor.doubling}
 Let $R$ be a rational map with a critical point $c \in J(R)$ and
 $\chi_-(R(c) \geq 0$. Assume
 \[
\operatorname{sign}(R'(R^n(R(c)))) =
\frac{R'(R^n(R(c)))}{|R'(R^n(R(c)))|}
\]
 is the doubling sequence on symbols $\{-1, 1\}$.
 Then $R$ is unstable.
\end{corollary}

Throughout the paper, we also discuss examples and applications involving
rotational domains, rational circle endomorphisms, invariant line fields,
and Feigenbaum dynamics, illustrating how automorphic measures naturally
arise from the asymptotic behavior of derivatives along critical orbits and
from the geometry of the associated cocycles.

\subsection{Structure of the paper}

The paper is organized as follows.

In Section~2, we introduce the general framework of automorphic measures,
quasi-invariant measures, cocycles, and skew-product extensions associated
with rational maps. We review the necessary background on conservative and
dissipative dynamics, measurable cocycles, transfer operators, and deformation
spaces of rational maps. In particular, we discuss the pullback and pushforward
actions on differential forms with measurable coefficients, the associated Ruelle and Beltrami operators,
and the relation between invariant functionals and quasiconformal deformations.

Section~3 is devoted to the proofs of the main results. We first study
pseudoconformal and unimodular measures, and establish their connections with Lyapunov exponents, recurrent dynamics, and quasiconformal rigidity. In particular, we prove Theorem \ref{th.Uno}, which characterizes hyperbolicity in terms of the nonexistence of pseudoconformal measures with negative exponent, and then deduce
Corollary \ref{cor.Rhyper}, which relates such measures to invariant line fields and structural stability.

We then establish the main instability criterion of the paper. Theorem \ref{th.Dos}
shows that a weakly dissipative rational map admitting a
$(-1,1)$-unimodular measure must be $q$-unstable for some integer $q \geq 1$.
As a consequence, Corollary \ref{cor.sigma} shows that when the corresponding eigenvalue
is equal to $1$, the map is structurally unstable. These results generalize
the instability criteria obtained previously through summability methods and
connect automorphic measures with transverse directions in deformation spaces.

A substantial part of the paper is devoted to constructing
unimodular measures and vector fields from ergodic and
combinatorial conditions, including conservative cocycles,
bounded recurrence, and bounded velocity of arguments.

\subsection*{Acknowledgment}

The authors are deeply grateful to M.~Lyubich for numerous inspiring
discussions and suggestions that influenced the development
of this work.

\section{Preliminaries}

We assume the reader is familiar with the basic notions of the dynamics of
rational maps, quasiconformal theory, and Teichmüller spaces of Riemann
surfaces and rational maps; see \cite{GardLakic}, \cite{LyubichDynRatTran}, and \cite{McMSull}.

\subsection{Cocycles and skew-products}

Given a Riemann surface $S$, let $T(S)$ denote the holomorphic tangent
bundle of $S$, and let $UT(S) \subset T(S)$ denote the ``holomorphic unit tangent
bundle" of $S$, namely, the collection of all norm $1$ vectors (directions) in $T(S)$ .
 
If $R \colon \overline{\C} \to \overline{\C}$ is a holomorphic
endomorphism, then $R$ acts on $T(\overline{\C})$ naturally as follows:

\[
\tilde{R} \colon T(\overline{\C}) \to T(\overline{\C}),
\qquad
\tilde{R}(z, v) \mapsto \big(R(z),\, R'(z) \, v\big),
\]
in suitable coordinates.

Since $R$ is not injective, in order to construct $\widehat{R}$,
the induced action of $R$ on $UT(\overline{\C})$, we must
remove some finite subsets from
$\overline{\C}$. More precisely, if

\[
S_2 = \overline{\C} \setminus V(R),
\qquad
S_1 =R^{-1}(S_2),
\]
then the action
\[
\widehat{R} \colon UT(S_1) \to UT(S_2)
\]
is given by
\[
\widehat{R}(z, v) =
\left( R(z),\, \frac{R'(z)}{\lvert R'(z) \rvert} \, v \right)
\]
in suitable coordinates, where $UT(S_1)$ and $UT(S_2)$ are the unit tangent
bundles over $S_1$ and $S_2$, respectively.

If
\[
T_0(\overline{\C}) =
T(\overline{\C}) \setminus \{(z, 0) : z \in \overline{\C}\}
\]
is the complement of the zero section in $T(\overline{\C})$, and
\[
\pi_0 \colon T_0(\overline{\C}) \to UT(\overline{\C}),
\qquad
\pi_0(z, v) = \left(z,\, \frac{v}{\lvert v \rvert}\right)
\]
is the projection  in suitable coordinates, then
\[
\pi_0 \circ \tilde{R} = \widehat{R} \circ \pi_0.
\]

It is convenient to regard $\widehat{R}$
as the action of $R$ on the directions at each point $z \in \overline{\C}$
outside $C(R)$.

Let $A$ be a locally compact abelian group equipped with its Haar
measure $m$. Then a measurable function $f \colon X \to A$ defines a
measurable extension of $e$ to a skew-product endomorphism
\[
e_f \colon X \times A \to X \times A,
\qquad
e_f(x, a) = \big(e(x),\, f(x) \cdot a \big),
\]
equipped with the complex measure $\tilde{\mu} = \mu \times m$.

In this situation, $f$ generates a measurable cocycle with respect
to $e$ of the form
\[f(n, z) \colon \mathbb{N} \times \bC \to A
\]
given by
\begin{align*}
f(0, z)     &= 1, \\
f(1, z)     &= f(z), \\
f(n + m, z) &= f(n, z) \, f\big(m,e^n(z)\big),
\end{align*}
such that
\[
e_f^n(z, v) = \big(e^n(z),\, f(n, z) \cdot v\big).
\]
Since $f$ defines the cocycle $f(n, \cdot)$ uniquely, it will be convenient to identify a given
cocycle $a(n,z)$ with its generator $a(z):=a(1 ,z)$. Two cocycles $a(z)$
and $b(z)$ are \textit{cohomologous} whenever there exists a
measurable function
\[
h \colon X \to A
\]
such that
\[
a(z) = b(z) \cdot h\big(e(z)\big) \cdot \big(h(z)\big)^{-1},
\]
and such an $h$ is called the \textit{transfer function} between the
cocycles $a$ and $b$.

A measurable cocycle $f$ is a \textit{coboundary} whenever $f$
is cohomologous to the constant cocycle
\[
b(x) = 1_A,
\]
where $1_A$ is the unit in the group $A$.

For example, if  $A = \C^*$ is the multiplicative group, $e = R$ is a
rational map, and $f= R'$, then $R_{R'} = \tilde{R}$ on
$S_1 \times \C^* \subset T(S_1).$ Similarly,  if
$A = \mathbb{S}^1$ and $f = \frac{R'}{\lvert R' \rvert}$, then the measurable
skew-product $R_f$ (which is the $R_1$ defined in the introduction),
coincides with $\widehat{R}$, the action of $R$ on $UT(S_1)$, the
unit tangent bundle of $S_1$.

From now on, the abelian group $A$ will be either a multiplicative
subgroup of $\mathbb{C}^*$ or an additive subgroup of $\C$, and the
associated cocycles will be called either multiplicative or
additive, respectively.

We
associate to every multiplicative cocycle $f$ the additive
cocycle
\[
\log f = \ln \lvert f \rvert + i \, \Arg(f).
\]
Conversely, for every additive cocycle  $g$, the function
$\exp(g)$ defines a multiplicative cocycle. The associated
skew-products are related as follows: if
\[
\pi(z, x) = \big(z, \exp(x)\big),
\]
then
\[
e_{\exp(g)} \circ \pi = \pi \circ e_g.
\]

For a rational map $R$ and the multiplicative cocycle $f = R'$, the
function
\[
f_{(s, t)}(z) = s \, \ln \lvert R'(z) \rvert + i \, t \, \Arg \big(R'(z)\big)
\]
is an additive cocycle for every pair of real numbers $s$ and $t$.
\begin{remark}
 Note that  the function $f_{(s,t)}$  is also an additive cocycle for $s,t\in \C$. Thus, the corresponding notion of an $(s,t)$-automorphic measure can also be defined.
\end{remark}
Therefore, we have the following fact:

\medskip

\textbf{Fact.} \emph{A rational map $R$ admits an $(s, t)$-automorphic measure $\mu$
with eigenvalue $\lambda$ whenever the multiplicative skew-product $R_g$, with
\[
g= \exp \big(-f_{(s, t)}\big)/\lambda,
\]
 admits an invariant $\sigma$-finite measure equivalent to $\tilde{\mu}$.}
\medskip

Recall that a measurable set $A \subset X$ is $\mu$-\emph{wandering}
($\mu$-\emph{weakly wandering}) for the endomorphism $e$ and a complex
quasi-invariant measure $\mu$ if the backward orbit $\{e^{-n}(A)\}$ consists of (contains infinitely many of) pairwise disjoint elements modulo $ \operatorname{var}(|\mu|)$. Then,  modulo $\operatorname{var}(|\mu|)$, the sets
 \[
D(e) = \bigcup \{\, W \,  \mu\text{-wandering} \,\},
\qquad
WD(e) = \bigcup \{\, W\,  \mu\text{-weakly wandering} \,\}
\]
are called the \textit{dissipative} and \textit{weakly dissipative}
sets for $e$ with respect to $\mu$, respectively.  The complement
$X \setminus D(e)$ is called the \textit{conservative} set for $e$
with respect to $\mu$. Accordingly, $\mu$ will be called conservative
or dissipative for $e$, whenever $\operatorname{var}{(|\mu|)}(D(e)) = 0$ or $\operatorname{var}{(|\mu)|}(X \setminus D(e)) = 0$,
respectively. If the measure is understood, then the endomorphism $e$
will be called either conservative or dissipative, respectively.
 
The set $SC(e) = X \setminus WD(e)$ is called the \textit{strongly
conservative} set with respect to $\mu$  and plays a crucial role, as the following lemmas show
(see \cite{Aaronson}, \cite{CMFixedETDS}  and \cite{Krengel}).
 
\begin{lemma} Let $e \colon X \to X$ be a non-singular endomorphism with
respect to a quasi-invariant probability measure $\mu$ with Radon--Nikodym
derivative $\omega$. Then:

\begin{enumerate}
 
\item $\mu\big(SC(e)\big) > 0$ if and only if there exists an
invariant measure $\nu$ absolutely continuous with respect to $\mu$.
In this situation, the Radon--Nikodym derivative $\frac{d\nu}{d\mu} = 0$
$\mu$-almost everywhere on $WD(e)$.
 
\item For every $\epsilon > 0$ there exists a wandering set
$W$, such that
\[
\mu\!\left(D(e) \setminus \bigcup_i e^{-i}(W)\right) \leq \epsilon.
\]

\item If $e$ is a bijection, then there exists a wandering set
$W$ with
\[
D(e) = \bigcup_{i=-\infty}^\infty e^i(W)
\quad \bmod \, \mu.
\]

\item If $e$ is as in \textup{(3)} and $\omega(n, x)$ is the cocycle with
$\omega(1,x)=\omega(x)$, then
\[
D(e) = \left\{\, x : \sum_{n \geq 1} \omega(n, x)
\ \text{is absolutely convergent} \,\right\}
\]
$\mu$-almost everywhere.

\end{enumerate}
\end{lemma}

According to \cite{LyuTypical}, \cite{McMullenbook} (see also \cite{CMFixedETDS}),  more can be said for rational maps.
The following dichotomy lemma is  part of  Lemma 4 in \cite{CMFixedETDS}
\begin{lemma}\label{lem.SC}
For a rational map $R$, either
\begin{itemize}
\item $SC(R)\subset P(R)$ or
\item $SC(R)=\overline{\C}$.
\end{itemize}
Additionally, in the latter case, if $R$ admits an invariant Beltrami differential then $R$ is M\"obius conjugated to a flexible Latt\'es map.
\end{lemma}

As an application of the discussion above, in the dissipative case,
we have the following statement, which is  based on Sullivan's argument for dissipative actions of groups on measurable sets.
 
\begin{proposition}
Let $\nu$ be a $(-1)$-pseudoconformal measure for a rational map
$R$. If $X = \supp(\nu)$ and the restriction $R \colon X \to X$ is
injective with $\nu \big(D(R) \big) > 0$, then for every complex number
$\lambda$ with $\lvert \lambda \rvert = 1$ and real number $t$, there exists a
function $\psi \in L_\infty(X, \nu)$
such that
\[
\psi\big(R(z)\big)
\left( \frac{R'(z)}{\lvert R'(z) \rvert} \right)^{t}
= \lambda \cdot \psi(z),
\]
$\mu$-almost everywhere on $D(R)$. In particular, $R$ admits a
unimodular measure $\rho$ supported on $X$ with eigenvalue
$1$ and $|\rho| = \nu$.
\end{proposition}

\begin{definition}
Let $e$ be a nonsingular endomorphism with respect to a quasi-invariant measure $\nu$, and let $G$ be a locally compact abelian group with Haar measure $m$. We say that a cocycle $f:X\to G$ is conservative or dissipative if the skew-product $e_f$ is conservative or dissipative, respectively, with respect to $\widetilde\nu=\nu\times m$.
\end{definition}

The proposition above concludes the dissipative case; we
now proceed to the conservative case. The following lemma
appears in \cite{Aaronson}.

\begin{lemma}
Let $e \colon X \to X$ be a non-singular ergodic endomorphism
with respect to a probability quasi-invariant measure $\nu$,
and let $f \colon X \to \mathbb{R}$ be a measurable additive cocycle.
\begin{enumerate}
\item If $\tilde{\nu}\big(X \setminus D(e_f)\big) > 0$, then
\[
X \setminus D(e_f) = X \times \mathbb{R} \quad \bmod(\tilde{\nu}).
\]
\item If $\tilde{\nu}\big(SC(e_f)\big) > 0$, then
\[
SC(e_f) = X \times \mathbb{R} \quad \bmod(\tilde{\nu}).
\]
\end{enumerate}

\end{lemma}

Hence, for every real-valued cocycle $f$, the skew-product $e_f$
is either conservative or  dissipative, but not both. Note that $e_f$ is
conservative whenever $f$ is a coboundary.  The following
results give sufficient conditions for a cocycle to be a coboundary.

\begin{proposition}\label{pr.absolutelycontinuous}
 Let $e \colon X \to X$ be a non-singular ergodic endomorphism with
 respect to a quasi-invariant probability measure $\nu$ with
 Radon--Nikodym derivative $\omega$, and let
 $f \colon X \to \mathbb{R}$ be a measurable additive cocycle. Then:
\begin{enumerate}

\item $f$ is a coboundary if and only if $e_f$ admits a
$\sigma$-finite invariant measure absolutely continuous with
respect to $\tilde{\nu}$.

\item If $f \in L_\infty(X, \nu)$, then $f$ is a coboundary if
and only if the partial sums
\[
s_n(x) = \sum_{i = 0}^{n - 1} f\big(e^i x\big)
\]
are uniformly bounded $\nu$-almost everywhere.

\item If $f = \log \big(\omega\big)$, then $f$ is a coboundary
if and only if there exists a $\sigma$-finite invariant measure $\rho$
absolutely continuous with respect to $\nu$.
\end{enumerate}

\end{proposition}

Part~\textup{(1)} appears in \cite[Theorem~3.1.3]{Aaronson}.
Part~\textup{(2)} is due to Lin and Sine~\cite{Lin1983}.
Part~\textup{(3)} follows from the definition of a coboundary together
with \cite[Theorem~5.5 and Theorem~10.5]{SchmidtRig} and
\cite[Theorem~2]{SchmidtRecurrence}.
 
 \begin{theorem}\label{th.Schmidt}
Let $e \colon X \to X$ be a non-singular automorphism with respect to
an ergodic probability quasi-invariant measure $\rho$, and let
$f \colon X \to \mathbb{R}^n$ be a conservative cocycle. Then,
for every measurable set $B \subset X$ with $\rho(B) > 0$, there exists
a non-atomic $\sigma$-finite invariant measure $\nu_B$ on $X$ with
$\nu_B(B) > 0$ such that $\nu_B$ is $e$-ergodic and $f$ is a coboundary
for $e$ with respect to $\nu_B$.
 \end{theorem}

Hence, every conservative cocycle is a coboundary on a suitable invariant
subset admitting an invariant non-atomic measure.  The next proposition is a corollary of results in \cite{Atkinson1976}
(see also \cite{SchmidtRig}) and Theorem~2 of \cite{SchmidtRecurrence}. Recall that Atkinson's theorem asserts that for an ergodic
measure-preserving transformation and an integrable
$\mathbb{R}^d$-valued cocycle with zero mean, the corresponding skew-product is recurrent. In our setting, recurrence is equivalent to
conservativity, so the zero-mean condition in each coordinate
characterizes conservative $\mathbb{R}^d$-valued cocycles.

In general, determining whether a cocycle is conservative is a non-trivial task. However, for quasi-invariant measures that are equivalent to finite invariant measures, this condition takes  the following simple form.

\begin{proposition}\label{pr.Atkinson}
Let $e$ be a non-singular automorphism with respect to an ergodic
probability quasi-invariant measure $\rho$, and let
$f = (f_1, \ldots, f_d)\colon X \to \mathbb{R}^d$ be a measurable cocycle, where each
$f_i \colon X \to \mathbb{R}$ is $\rho$-integrable.
Let $\nu$ be an invariant probability measure equivalent to $\rho$.
Then the cocycle $f$ is conservative with respect to $\mu$ if and only if
\[
\int_X f_i \, d\rho = 0 \quad \text{for } i = 1, \ldots, d.
\]
\end{proposition}

\subsection{Pull-back and pushforward actions on measurable
section spaces}

Let $S$ be a Riemann surface, and define
\[
T_n(S) =
\begin{cases}
(T^{*} (S))^{\otimes n}, & \text{if } n \geq 0, \\
(T(S))^{ \, \otimes |n|}, & \text{if } n < 0,
\end{cases}
\]
where $T(S)$ and $T^*(S)$ are the holomorphic tangent and
the cotangent bundles, respectively. For $n<0$ we interpret
$(T(S))^{ \, \otimes |n|}$ as the symmetric tensor product of $|n|$ copies of $T(S)$.

For $m, n \in \mathbb{Z}$, let
\[
A_{m, n} = T_m(\overline{\C}) \otimes \overline{T}_n(\overline{\C})
\]
be the tensor (symmetric) product of the holomorphic and the antiholomorphic tangent
bundles. Here $\overline{T}_n(S)$ denotes the antiholomorphic counterpart
of $T_n(S)$.

If $\pi \colon S_1\to S_2$ is a holomorphic branched covering map,
then $\pi$ induces the pullback
\[
\pi^* \colon A_{m, n}(S_2) \to A_{m, n}(S_1)
\]
and the pushforward
\[
\pi_* \colon A_{m, n}(S_1) \to A_{m, n}(S_2)
\]
operators, whenever $(m,n)\neq (0,0)$.

Let $F_{(m, n)}(S)$ be the space of measurable sections of
$A_{(m,n)}(S)$, which is naturally identified with the space of
differentiable $(m, n)$-forms $\omega$ with measurable coefficients.
In suitable coordinates,
\[
\omega(z) = \varphi(z) \, dz^m \, d\bar{z}^n,
\]
where $\varphi(z)$ is a complex-valued measurable function.
Hence, $\pi$ acts as above on $F_{(m, n)}(S_i)$ in the following way.

For $\omega \in F_{(m, n)}(S_1)$ with local representation
$\omega(z) = \varphi(z) \, dz^m \, d\bar{z}^n$, the pushforward is
\begin{align*}
\pi_{*\, (m, n)}(\omega)(z)
&= \left[ \sum_{x \in \pi^{-1}(z)}
    \frac{\varphi(x)}{ \big(\pi'(x)\big)^{m} \, \big(\overline{\pi'(x)}\big)^{n}}
   \right] dz^m \, d\bar{z}^n,
\end{align*}
and for  $\omega \in F_{(m, n)}(S_2)$,  the pullback is
\begin{align*}
\pi^*_{(m, n)}(\omega)(z)
&= \varphi(\pi(z) \,\big(\pi'(z)\big)^m \, \big(\overline{\pi'(z)}\big)^{\,n} \,
   dz^m \, d\bar{z}^n.
\end{align*}

The indices $m$ and $n$ may also assume negative values, in which case
$T_n(S)$ denotes the symmetric tensor powers of the tangent bundle $T(S)$,
as defined earlier. The pushforward and pullback actions described above
thus extend to measurable $(m, n)$-forms of arbitrary type, thereby
generalizing the usual holomorphic and antiholomorphic actions to
the measurable category.

If $m = n = 0$, the pullback
\[
\pi^*_{(0, 0)}(\phi) = \phi\circ \pi
\]
is known as the \textit{Koopman operator} with respect
to $\pi$ acting on the space of functions. However, the corresponding pushforward operator is defined only when $deg(\pi)$ is finite and is given by the formula:
\[
\pi_{*(0,0)}(\phi)(z) :=  \sum_{x \in \pi^{-1}(z)}
    \phi(x) .
\]

Now let  $\pi$ be a rational map $R$. The operator  $\frac{1}{deg(R)}R_{*(0,0)}$ was first introduced into holomorphic dynamics by Lyubich in \cite{LyubichErgodicTheory}. We call $B_R:=R^*_{(-1,1)}$, the pullback operator on $(-1,1)$-forms, the Beltrami operator associated with $R$.
The Beltrami operator acts on the  coefficients of elements in $F_{(-1, 1)}(\overline{\C})$,
equipped with the $L_\infty$-norm, as a continuous linear endomorphism of norm one.

In this situation, one can canonically identify elements of
$F_{(-1, 1)}(\overline{\C})$ with their coefficients. Hence, the open unit ball
\[
U_1 \subset \operatorname{Fix}(B_R) \subset L_\infty(\C)
\]
in the fixed-point space of the Beltrami operator is called the space of
\textit{invariant Beltrami differentials for the map $R$}. By the
measurable Riemann mapping theorem, this space parameterizes all
quasiconformal  (non-conformal)  deformations of $R$.

The pushforward operator $R_{* \, (2,0)}$ is a linear endomorphism of
$F_{(2,0)}(\overline{\C})$ that preserves the subspace of measurable $(2,0)$-forms
absolutely integrable over $\overline{\C}$. In other words,
$R_{* \, (2,0)}$ restricted to  absolutely integrable elements of
$F_{(2,0)}(\overline{\C})$ can be treated as a continuous endomorphism of
$L_1(\C)$ with unit norm. Moreover, its dual operator acting
on the dual space $L_\infty(\C)$ is the Beltrami operator
$B_R$, defined above (see \cite{MakRuelle}).

The operator
\[
R_* := R_{*\, (2,0)}
\]
is known as the \textit{Ruelle--Perron--Frobenius transfer operator}
for the dynamical system defined by $R$ (for example, see \cite{CMFixedETDS}). The operator
\[
\lvert R_* \rvert := R_{*\, (1, 1)}
\]
represents the modulus of the operator $R_*$ (see  \cite{Krengel} and also \cite{CMFixedETDS}) and it is known as
the \textit{Perron--Frobenius transfer operator} associated with $R$.

The fixed points of $ R_* $ in $L_1(\C)$ uniquely
determine complex-valued finite invariant measures absolutely continuous
with respect to the Lebesgue measure. Moreover, for every
$\phi \in L_1(\C)$ we have
\[
\lvert R_*(\phi) \rvert \;\leq\;
\lvert R_* \rvert \big(\lvert \phi \rvert\big).
\]

Let $S = \overline{\C} \setminus P(R)$, and let
$A_1(S) \subset L_1(\C)$ be the subspace of functions holomorphic on $S$.
The Ruelle operator $R_*$ leaves $A_1(S)$ invariant \cite{MakRuelle}.
If $T$ denotes the dual operator to $R_*$ acting on the dual space
$A_1^*(S)$ of $A_1(S)$, then $T$ can be regarded as a generalization of
the infinitesimal Thurston pullback operator defined for
postcritically finite rational maps in \cite{ DHTop}, see also \cite{CMFixedETDS}.

We will say that an element $\phi \in A_1(S)$ is \textit{rational}
if $\phi$ can be represented by a rational function. For example, if
$a_1, a_2, a_3 \in P(R)$ are three distinct points, then the rational function
\[
\frac{1}{(z - a_1)(z - a_2)(z - a_3)}
\]
belongs to $A_1(S)$. By the Bers density theorem \cite{GardLakic},
the set of rational elements forms a dense subspace of $A_1(S)$.

From now on, for convenience, we will assume that the postcritical
set $P(R)$ is bounded and that every rational map fixes the points
$\{0, 1, \infty\}$. We therefore consider the normalized space
\[
A_1(R) =
\{\, \phi \in L_1(\C) \mid
\phi \text{ is holomorphic on }
\C \setminus (P(R) \cup \{0, 1\}) \,\}.
\]
The function
\[
\gamma_a(z) = \frac{a(a - 1)}{z(z - 1)(z - a)}
\]
belongs to $A_1(R)$.

Recall that $R$ has critical relations if either
\begin{itemize}
\item there exist two critical points $c_1, c_2 \in C(R)$ and integers
$n, m \ge 0$ such that $R^n(c_1) = R^m(c_2)$, or
\item there exists a critical point $c$ with local degree at least $3$.
\end{itemize}
By~\cite{MakRuelle}, we have:

\begin{proposition}\label{pr.MakRuelle}
Let $R$ be a rational map fixing $\{0, 1, \infty\}$ with no
critical relations. Then the linear span of
$\{\gamma_a(z)\}_{a \in P(R)}$ is dense in $A_1(R)$.
If $a \in P(R)$ is not a critical point, then
\[
R_*(\gamma_a) \in A_1(R)
\quad\text{and}\quad
R_*(\gamma_a(z)) =
\frac{\gamma_{R(a)}(z)}{R'(a)}
+ \sum_{c_i \in C(R)} b_i \, \gamma_a(c_i) \, \gamma_{R(c_i)}(z),
\]
where $b_i = \frac{1}{R''(c_i)}$.
\end{proposition}

\subsection{Deformation spaces of rational maps}

Following Hurwitz's classification of branched coverings in real dimension two
\cite{Hur}, as well as \cite{Berstein}, we say that
two branched coverings
\[
R \colon S \to S'
\quad \text{and} \quad
Q \colon W \to W'
\]
between Riemann surfaces are \textit{Hurwitz equivalent}
if there exist homeomorphisms $\phi$ and $\psi$ making
the following diagram  commutative:
\begin{displaymath}
    \xymatrix{
        S \ar[d]_{R} \ar[r]^{\psi} &
        W \ar[d]^{Q} \\
        S' \ar[r]^{\phi} & W'
    }
\end{displaymath}

We say that the coverings $R$ and $Q$ are
\textit{Hurwitz quasiconformally equivalent} if the homeomorphisms
$\phi$ and $\psi$ in the diagram above can be taken to be quasiconformal.

For a branched covering $R \colon S \to S'$, let $HW(R)$
(respectively, $HQ(R)$)  be the collection of all branched
coverings $Q \colon W \to W'$ that are Hurwitz (respectively, Hurwitz quasiconformally) equivalent to $R$.
The space $HW(R)$ (respectively, $HQ(R)$) is called
the \textit{Hurwitz class} (also known as  \textit{Hurwitz space})
(respectively, \textit{Hurwitz quasiconformal class}, or
\textit{Hurwitz quasiconformal space}) of the covering $R$.

By the Teichmüller theorem, the homeomorphisms $\phi$ and $\psi$ in
the diagram above can be chosen quasiconformal whenever $R$ and
$Q$ are holomorphic coverings over geometrically finite Riemann surfaces. 

Let $Rat_d$ be the set of holomorphic endomorphisms of the Riemann
sphere of degree $d$. By coefficient parameterization, $Rat_d$ forms
an open, everywhere dense subset of $\mathbb{CP}^{2d + 1}$.
For a rational map $R \in Rat_d$, let
\[
H(R) = HW(R) \cap Rat_d = HQ(R) \cap Rat_d
\]
be the set of all rational maps $Q$ that are Hurwitz equivalent
to $R$ equipped with the compact--open topology. We call the
space $H(R)$ the \textit{holomorphic Hurwitz class} of $R$,
or simply the \textit{Hurwitz class} of $R$.

Note that for holomorphic endomorphisms of the complex plane, the holomorphic Hurwitz class was used to parameterize families of maps in the so-called Eremenko--Lyubich class \cite{EreLyub} of entire (or even meromorphic) endomorphisms. For holomorphic branched coverings of finite
type between generic hyperbolic surfaces, the topology and geometry of
these classes were studied in \cite{Epstrans}. For rational
maps, see also (\cite{CabMakSie}).
For a nonconstant rational map $R$, let
\[
S_R = \overline{\C} \setminus \big(V(R) \cup \{0,1,\infty\}\big),
\qquad
S'_R = R^{-1}(S_R).
\]
Then $S_R$ is a hyperbolic Riemann surface, and $R:S'_R\to S_R$ is an unbranched covering.

Let $f \colon S_R \to W$ be a representative of a point $[f]$ in the
Teichmüller space $\operatorname{Teich}(S_R)$ fixing the points $\{0,1,\infty\}$. Let
$\mu_f$ be the Beltrami coefficient of $f$, and $\nu$ be its pullback
by $R$ to $S'_R$ given by
\[
\nu = B_R(\mu_f) = (\mu_f \circ R)\,\frac{\overline{R'}}{R'}.
\]
Let $h_f$ be the (normalized) quasiconformal homeomorphism solving the
Beltrami equation with coefficient $\nu$ and fixing $\{0, 1, \infty\}$,
and set $W' := h_f(S'_R)$. Define the map
$\tau \colon \operatorname{Teich}(S_R) \to H(R)$ by requiring that $\tau([f])$ is the
unique rational map making the following diagram commutative:

\[
\begin{array}{c}
    \xymatrix{
        S'_R \ar[d]_{R} \ar[r]^{h_{f}}  & W' \ar[d]^{\tau([f])} \\
        S_R  \ar[r]^{f}                 & W
    }
\end{array}
\]

The map $\tau$ is well-defined, continuous, and  even analytic in the respective complex structures.   We call the space
\[
RH(R) := \tau\big(\operatorname{Teich}(S_R)\big)
\]
the \textit{reduced Hurwitz space} of $R$.
Hence, $H(R)$ is the bi-orbit of $RH(R)$ with respect to the left and/or right actions of the M\"obius group.
As shown in \cite{CabMakSie}, the map $\tau$ is the universal
branched covering map and therefore
\[
\dim RH(R) = \dim \operatorname{Teich}(S_R) = \#(V(R)).
\]
For a generic rational map $R$ of degree $d$, both $H(R)$ and $RH(R)$
are suborbifolds of $Rat_d$ of dimensions $2d + 1$ and
$2d - 2$, respectively.

\begin{definition}Let $F \subset Rat_d$ be a submanifold. A map $R \in F$ is called
 \textit{structurally stable} in $F$ (or \textit{$F$-stable} for short)
 if there exists a neighborhood $U \subset F$ such that, for every
 $g \in U$, there exists a quasiconformal automorphism $h_g$ of $\widehat{\C}$
 with
 \[
 g = h_g \circ R \circ h_g^{-1},
 \]
 and such that the map $g \mapsto h_g$ depends continuously on $g$.
 If $F = Rat_d$, then $R$ is called \textit{structurally stable}.
 \end{definition}

\section{Proofs}\label{proofs}

\subsection{Pseudoconformal measures}

\textit{Proof of Theorem \ref{th.Uno}.} To show Theorem \ref{th.Uno}, we first state two auxiliary lemmas.

We begin with the simplest case in which unimodular measures exist.

\begin{lemma}\label{lem.simptransversals}
 A rational map $R$ admits a unimodular measure $\rho$ with
 $(s,t)= (-1,1)$ whenever at least one of the following conditions holds:
 \begin{enumerate}
 \item The Julia set $J(R)$ possesses a neutral periodic point.
 \item The Fatou set $F(R)$ contains a rotational domain.

\end{enumerate}
 
\end{lemma}
\begin{proof}
 Let $x$ be a neutral periodic point in $J(R)$ with period $p$ and multiplier
 $\lambda = e^{2\pi i \theta}$ with $\theta \in [0,1]$. Let $\omega$  be a
$p$-th root of $\lambda$. Then the measure
\[
\rho_\omega = \delta_x + \frac{\omega}{R'(x)} \delta_{R(x)} +
 \frac{\omega^2}{(R^2)'(x)} \delta_{R^2(x)} + \dots +
 \frac{\omega^{p-1}}{(R^{p-1})'(x)} \delta_{R^{p-1}(x)}
\]
 satisfies the formula
 \[
\int \frac{\phi(R)}{R'} \, d\rho_\omega = \frac{1}{\omega} \int \phi \, d\rho_\omega,
\]
 and hence it is the desired unimodular measure with eigenvalue $1/\omega$.
 
 Assume that there exists an invariant rotational domain $U\subset F(R)$,
 and let $\mathcal{L}$ be the leaf of the invariant foliation in $U$. Let
 $h \colon \mathcal{L} \to \mathbb{S}^1$ be a conformal map conjugating
 $R$ with
 the rotation $z \mapsto \lambda z$, $\lambda \in \mathbb{S}^1$. Let $\eta$
 be the one-dimensional Lebesgue measure on $\mathbb{S}^1$. Then the functional
\[
L(\phi) = \int_{\mathbb{S}^1} \frac{\phi(h^{-1}(x))}{(h^{-1})'(x)} \, d\eta(x)
\]
 is continuous on $C(\mathcal{L})$, the space of continuous functions on
$\mathcal{L}$. A straightforward computation gives
\[
L\!\left(\frac{\phi(R)}{R'}\right) = \frac{1}{\lambda} L(\phi).
\]
 Now we apply the Riesz representation theorem to get the desired measure $\rho$.

Finally, assume that $U$ has period $p$ and that $R^p$ is conjugate with
$z \mapsto \lambda z$. As above, we construct a measure $\rho_1$ in $U$
for the map $R^p$ and let $\rho_i$ be the pullback of $\rho_1$ by $R^{i-1}$ along the respective orbit of $U$. Then the desired measure is
\[
\rho = \rho_1 + \frac{1}{\omega}\rho_2 + \dots +
 \frac{1}{\omega^{p-1}} \rho_{p},
\]
where $\omega$ is a $p$-th root of $\lambda$.
\end{proof}

\begin{definition}

Let $R$ be a rational map $R$ and $X_i$ be a periodic cycle of $R$, of period $p_i$ and multiplier $\lambda_i$. We call the family of periodic cycles  $\{X_i\}$  a \emph{degenerating sequence of cycles} if
\[
\lim_{i \to \infty} \chi(X_i) = 0,
\]
where $\chi(X_i) = \tfrac{1}{p_i} \ln |\lambda_i|$ is the Lyapunov exponent of
$X_i$.
\end{definition}

\begin{definition}
For a rational map $R$ and a sequence of points $\{a_i\}\subset \overline{\C}\setminus C(R)$ such
that $a_i = R(a_{i+1})$, the \textit{lower backward Lyapunov exponent}
of $a_0$ with respect to $\{a_i\}$ is
\[
b\chi_-(a_0) = \liminf_{n \to \infty} \frac{1}{n}
\ln \left| \frac{1}{(R^n)'(a_n)} \right|.
\]
Similarly, the \textit{upper backward Lyapunov exponent} of $a_0$ with
respect to $\{a_i\}$ is defined as
\[
b\chi_+(a_0) = \limsup_{n \to \infty} \frac{1}{n}
\ln \left| \frac{1}{(R^n)'(a_n)} \right|.
\]
\end{definition}

\begin{lemma}\label{lm.Lyapunov}
Let $R$ be a rational map.
\begin{enumerate}
\item Let $z_0\in \C$ be a point with finite $\chi_-(z_0)$ with
respect to $R$. Then $R$ admits a probability $(s,0)$-automorphic
non-atomic measure $\rho_s$ with eigenvalue
$\lambda=\exp(s\chi_-(z_0))$, whenever  $s \leq 0$ and the derivative set of the forward orbit of $z_0$ is not finite.
 
\item Let $z_0 \in \C$ be a point with finite $\chi_+(z_0)$ with
respect to $R$. Then $R$ admits a probability $(s,0)$-automorphic
non-atomic measure $\rho_s$ with eigenvalue
$\lambda=\exp(-s \chi_+(z_0))$, whenever $s \geq 0$  and the derivative set of the forward orbit of $z_0$ is not finite.
 
\item If there exists a point $a_0$ with finite $b\chi_{-}(a_0)$
with respect to a suitable sequence $\{a_i\}$, then $R$ admits a non-atomic $(s,0)$-automorphic probability measure $\rho_s$ with
eigenvalue $\lambda = \exp(-s b\chi_-(a_0))$ whenever
$ s \geq 0$ and  the derivative set of $\{a_i\}$ is not finite.

\item If there exists a point $a_0$ with finite $b\chi_{+}(a_0)$
with respect to a suitable sequence $\{a_i\}$, then $R$ admits a nonatomic $(s,0)$-automorphic probability measure $\rho_s$ with
eigenvalue $\lambda = \exp(-s b\chi_+(a_0))$ whenever $s \leq 0$ and  the derivative set of $\{a_i\}$ is not finite.
 \end{enumerate}
\end{lemma}

\begin{proof}
 \textit{Part 1.} If $z_0$ satisfies the assumptions, then for $s \leq 0$ the measure
\[
\tilde{\rho}_{(\lambda,s)} =
 \sum_{n \geq 0} |(R^n)'(z_0)|^s \lambda^n \delta_{R^n(z_0)}
\]
 is finite with norm
 \[
\|\tilde{\rho}_{(\lambda,s)}\| =
 \operatorname{Var}(\tilde{\rho}_{(\lambda,s)}) =
 \sum_{n \geq 0} |(R^n)'(z_0)|^s |\lambda|^n,
\]
 for every $\lambda\in \C$ with
 \[
|\lambda| < \limsup_{n \to \infty} \bigl( |(R^n)'(z_0)|^{1/n} \bigr)^s
= \exp(s \chi_-(z_0)) = r_s.
\]
If
 \[
\rho_{(\lambda,s)} =
 \frac{\tilde{\rho}_{(\lambda,s)}}{\| \tilde{\rho}_{(\lambda,s)} \|},
\]
then
 \[
\int \phi(R)|R'|^s \, d\rho_{(\lambda,s)} =
 \frac{1}{\lambda}
 \left( \int \phi \, d\rho_{(\lambda,s)} -
 \frac{\phi(z_0)}{\|\tilde{\rho}_{(\lambda,s)}\|} \right).
\]

 We conclude that every accumulation point of the measures
$\rho_{(\lambda,s)}$, as $\lambda \to r_s$, is an
$(s,0)$-automorphic probability measure with eigenvalue
$1/r_s$, whenever the respective norm
$\|\tilde{\rho}_{(\lambda,s)}\| \to \infty$.
Otherwise, the series
\[
\sum_{n \geq 0} |(R^n)'(z_0)|^s t^n
\]
converges for $t \leq r_s$ and diverges for $t > r_s$. In other words, $\log(r_s)$ is a
critical exponent. In this situation, we use the ``mollifier sequence'' $\{b_n\}$ constructed in Lemma 3.1  in   \cite{DenkerUrbanski} and satisfying
the following properties:
\[
\lim_{n \to \infty} \frac{b_{n+1}}{b_n} = 1
\quad \text{and} \quad
\sum_{n \geq 0} b_n |(R^n)'(z_0)|^s t^n
\]
converges for $t < r_s$ and diverges for $t \geq r_s$.

Then, any accumulation point of the modified measures
\[
v_{(\lambda,s)} =
\frac{\sum_{n \geq 0} b_n |(R^n)'(z_0)|^s
\delta_{R^n(z_0)} \lambda^n}
{\sum_{n \geq 0} b_n |(R^n)'(z_0)|^s |\lambda|^n}
\]
is a non-atomic probability measure satisfying the required properties. Indeed,
\[
\left| \int \phi(R)|R'|^s \, dv_{(\lambda,s)} - \frac{1}{\lambda}\int \phi \, dv_{(\lambda,s)} \right|
\]
\[
= \frac{\left| \sum_{n \geq 1} \phi(R^n(z_0))
b_n |(R^n)'(z_0)|^s  \left(\tfrac{b_{n-1}}{b_n}-1\right) \lambda^n
- \frac{b_0 \phi(z_0)}{\lambda}\right|}
{\sum_{n \geq 0} b_n |(R^n)'(z_0)|^s |\lambda|^n} = (*).
\]

Now fix $\epsilon > 0$, $N > 0$, and $\lambda_0$ such that for all $n \geq N$
and $\lambda_0 \leq \lambda \leq r_s$, we have
\[
\left| \frac{b_{n-1}}{b_n} - 1 \right| < \epsilon
\]
and
\[
\frac{\left| \sum_{n=1}^N  b_n |(R^n)'(z_0)|^s
\left(\tfrac{b_{n-1}}{b_n}-1\right) \lambda^n -
\frac{b_0}{\lambda} \right|}
{\sum_{n \geq 0} b_n |(R^n)'(z_0)|^s  |\lambda|^n}
\leq \epsilon.
\]
Then
\begin{multline*}
(*) \le \epsilon \|\phi\|_\infty
 + \frac{\Bigl|
 \sum_{n=N}^\infty
   \phi(R^n(z_0))\, b_n |(R^n)'(z_0)|^s
   \left(\tfrac{b_{n-1}}{b_n}-1\right)\lambda^n
   - \frac{b_0 \phi(z_0)}{\lambda}
 \Bigr|}
 {\displaystyle\sum_{n \ge 0} b_n |(R^n)'(z_0)|^s |\lambda|^n}
\\[2mm]
\le 2\epsilon \|\phi\|_\infty.
\end{multline*}
This proves Part~1.

Part~2 follows by an analogous argument, with the appropriate modifications.

For Part~3 and Part~4,  suppose $a_0$ satisfies the assumptions. Then the measure
\[
 \tilde{\rho}_{(\lambda,s)}=\sum_{n\geq 0}\frac{\lambda^n \delta_{a_n}}{|(R^n)'(a_n)|^s}
\]
is finite with norm
\[
\|\tilde{\rho}_{(\lambda,s)}\| =\sum_{n\geq 0}\frac{|\lambda|^n }{|(R^n)'(a_n)|^s}
\]
for $\lambda\in \C$ with
\[
|\lambda|<\limsup\left(\frac{1}{|(R^n)'(a_n)|^{1/n}}\right)^s
=r_s(a_0)=\begin{cases}
\exp(-sb\chi_-(a_0)), \textnormal{ for } s\geq 0\\
\exp(-sb\chi_+(a_0)), \textnormal{ for } s\leq 0\
\end{cases}.
\]

If
\[
\rho_{(\lambda,s)}=\frac{\tilde{\rho}_{(\lambda,s)}}{\| \rho_{(\lambda,s)}\|},
\]
then
\[
 \int \phi(R)|R'|^s\, d\rho_{(\lambda,s)} = \lambda \int \phi \,  d\rho_{(\lambda,s)}+\frac{\phi(R(a_0))|R'(a_0)|^s}{\|\tilde{\rho}_{(\lambda,s)}\|}
\]
If necessary, using a mollifier sequence as above, the conclusion follows by passing to the limit $\lambda \to r_s(a_0)$.
\end{proof}

\begin{remark}
 When the derivative set of either the orbit of $z_0$ or $\{a_i\}$ is finite in  Lemma above, the arguments of the proof also yield an automorphic measure with the associated eigenvalues. However, this measure is atomic.
\end{remark}

\begin{theorem}\label{tm.pseudoconformal}
Let $s$ be a real number. Then the rational map $R$ admits an
$s$-pseudoconformal measure whenever at least one of the
following conditions holds:

\begin{enumerate}
 \item The Julia set $J(R)$ contains a neutral periodic point. 
 \item The Fatou set $F(R)$ contains a rotational domain.
 \item For  $s \leq  0$, there exists  a point $z_0$ with either
 $\chi_-(z_0) = 0$ or $b\chi_+(z_0) = 0$ with respect
 to a suitable sequence $\{a_n\}$.
 
\item For $s \geq 0$, there exists a point $z_0$ with either $\chi_+(z_0) = 0$ or $b\chi_-(z_0) = 0$ with respect to a suitable sequence $\{a_n\}$.
 
 \item  $R$ admits a degenerating sequence of cycles.
\end{enumerate}

\end{theorem}

\begin{proof}

\textit{Part 1.} Let $\lambda$ be the multiplier of a neutral periodic point,
so that $|\lambda| = 1$. Then the variation measure of the measure constructed
in Lemma~\ref{lem.simptransversals}(1) is the desired measure.
In this case, the measure is an atomic measure supported
on the neutral periodic cycle.

\textit{Part 2.} By Lemma~\ref{lem.simptransversals}(2), it suffices to establish
the existence of an $s$-pseudoconformal measure when $F(R)$ contains an
invariant rotational domain $U$. Let $L$ be a leaf of the invariant foliation
in $U$, and let $h : L \to \mathbb{S}^1$ be a linearization map. Then $h$ is a conformal
map conjugating $R$ with an irrational rotation
$r_\omega(z) = \omega z$. For every real $s$
define $T_s \colon C(L) \to C(\mathbb{S}^1)$ by
\[
T_s(\phi) = |(h^{-1})'|^s\cdot (\phi \circ h^{-1}).
\]
If $\eta$ is the one-dimensional Lebesgue measure on $\mathbb{S}^1$
and $T_s^*$ is the dual operator of $T_s$ on the corresponding dual space, then the
pullback $\rho_s = T_s^*(\eta)$ is the desired measure. Indeed,
\[
\int_L \phi(R) |R'|^s \, dT_s^*(\eta) =
\int_{\mathbb{S}^1} T_s\!\left(\phi(R)|R'|^s\right) d\eta =
\int_{\mathbb{S}^1}
\left( [\phi(R)|R'|^s] \circ h^{-1} \right) |(h^{-1})'|^s \, d\eta
\]
\[
= \int_{\mathbb{S}^1} (T_s \phi) \circ r_\omega \, d\eta
= \int_{\mathbb{S}^1} T_s(\phi) \, d\eta
= \int_L \phi \, dT_s^*(\eta).
\]

\textit{Parts 3 and 4.} These follow directly from Lemma~\ref{lm.Lyapunov}.

\textit{Part 5.} As in Lemma~\ref{lem.simptransversals}(1), let $X$ be a cycle of
period $p$ and multiplier $\lambda$, and let $\omega = \sqrt[p]{|\lambda|}$
be a $p$-th root of the multiplier. Then for real $s$ and $x_1 \in X$, the measure
\[
\rho_{s,X} = \frac{\delta_{s,X}}{\|\delta_{s,X}\|},
\]
\[
\delta_{s,X} = \delta_{x_1} + \frac{\omega^s}{|R'(x_1)|^s} \delta_{R(x_1)}
+ \frac{\omega^{2s}}{|(R^2)'(x_1)|^s} \delta_{R^2(x_1)} + \dots +
\frac{\omega^{s(p-1)}}{|(R^{p-1})'(x_1)|^s} \delta_{R^{p-1}(x_1)},
\]
is a probability measure satisfying
\[
\int \phi(R)|R'|^s \, d\rho_{s,X} = \omega^s \int \phi \, d\rho_{s,X}. \tag{*}
\]
Hence $\rho_{s,X}$ is an $(s,0)$-automorphic probability measure
with eigenvalue $\omega^s.$

Now let $X_i$ be a degenerating sequence of cycles of periods $p_i$ and
multipliers $\lambda_i$ for a rational map $R$.
Let $\{\rho_{s,X_i}\}$ be the corresponding sequence of
$(s,0)$-automorphic measures with eigenvalues
$|\lambda_i|^{s/p_i}$ constructed above.
By assumption, if $\rho_s$ is a $*$-weak accumulation point of
the measures $\rho_{s,X_i}$, then $\rho_s$ is a probability measure with
\[
\lim_{i \to \infty} |\lambda_i|^{s/p_i} =
\lim_{i \to \infty} \exp(s \chi(X_i)) = 1.
\]
We claim that $\rho_s$ is the desired measure. If $s>0$, then the equation $(*)$ implies the claim. For $s\leq 0$,
if
\[
\phi_i=\frac{d\rho_{s,X_i} \circ R^{-1}}{d\rho_{s,X_i}},
\]
then $\phi_i\circ R=|R'|$ $(\rho_{s,X_i})$-almost everywhere. Hence, the functions $\phi_i$ are continuous on $\overline{\C}\setminus V$ and, by assumption, converge uniformly on compact sets to a continuous function $\phi_0$, with
\[
 \phi_0(R)=|R'|.
\]

Then, for every $f$ continuous on $V$ and every $i$, we have
\[
 \int f(R) \, d\rho_{s,X_i} = |\lambda_i|^{s/p_i}\int f \phi_i d\rho_{s,X_i}.
\]
Taking limits in the preceding equality gives:

\[
 \int f(R) \, d\rho_s = \int f \phi_0 d\rho_s.
\]

This is equivalent to

\[
 \int f(R)\frac{1}{\phi_0(R)} \, d\rho_s = \int f(R) \frac{1}{|R'|} d\rho_s=\int fd\rho_s,
\]
 as claimed. \end{proof}
\begin{proposition}\label{pr.limpseudo}
Let $R$ be a rational map in $PM(d)$, then either $R$ admits a pseudoconformal measure or possesses a summable critical point. Both cases may occur.
\end{proposition}

\begin{proof}
 By definition and the topology on $Rat_d$, the map $R$ is a uniform limit of rational maps $R_i$ admitting a $-1$-pseudoconformal measures $\rho_i$.  Let $\rho$ be a *-weak accumulation of the measures $\rho_i$. Fix $\epsilon>0$ and let $\phi_\epsilon$ be the continuous function which is identically $0$ on an $(\epsilon/2)$-neighborhood of $V(R)$ and identically $1$ outside an $\epsilon$-neighborhood of $V(R)$. Then, for every natural $n\geq 0$ and a continuous function $f$ we have
\begin{equation}
  \int \frac{f(R^{n})\phi_\epsilon(R^{n})}
  {|(R^{n})'|}\,d\rho
  =
  \int f\phi_\epsilon\,d\rho.
  \tag{*}\label{eq:invariance}
\end{equation}

 If  $\rho(c)=0$ for all $c\in C(R)$, then as
 \[
  \frac{1}{|R'|}=\lim_{\epsilon \rightarrow 0} \frac{\phi_\epsilon}{|R'|}
 \] pointwise on $\overline{\C}\setminus(R^{-1}(V(R)))$ then $\frac{1}{|R'|}$ is integrable by Fatou's Lemma and $\rho$ is a -pseudconformal measure for $R$ by the Lebesgue dominated theorem for Borel measures. If there is $c\in C(R)$ such that $\rho(c)\neq 0$ then, every point in the forward orbit of $c$ is also an atom for $\rho$, and by  equation \ref{eq:invariance}, the point $c$ is a critical summable point \end{proof}

\begin{remark}
\leavevmode
 \begin{itemize}
 \item An invariant measure $\mu$ on the indifferent cycle $\{x_i\}$ is $s$-pseudo\-conformal for $s< 0$ if and only if $|R'(x_i)|=1$ for every element in $\{x_i\}$. Otherwise, these  kinds of measures are equivalent.

 \item It is possible that the $\omega$-limit set of a point
 $z_0 \in J(R)$ with a strictly negative Lyapunov exponent supports an
 $s$-pseudoconformal measure. For example, by following \cite{CMPVoronoi}
 and  using the N\"orlund--Voronoi summability method one can construct
 an $(s,0)$-pseudoconformal measure whenever the generating function
 $f(\lambda) = \sum \lambda^n |(R^n)'(z_0)|^s$ can be analytically
 extended to a function holomorphic  on a neighborhood of the
 open unit interval.
\end{itemize}
\end{remark}

A closed set $A \subset \overline{\C}$ is called \emph{hyperbolic} if it is
$R$-invariant and, in suitable coordinates, there exist $n \geq 1$ and a real
$t > 1$ such that $|(R^n)'(x)| \geq t$ for every $x \in A$. To conclude
the proof of Theorem~\ref{th.Uno}, we need the following fact.

\begin{lemma}\label{lm.hyperbolic}
Let $m_s$ be an $s$-pseudoconformal measure for a rational map $R$ with
$s < 0$. Then:
\begin{enumerate}
\item The support of $m_s$ is not a hyperbolic set.
\item $m_s(D) = 0$ whenever $D$ is a Fatou component in the grand orbit
of an attracting or superattracting periodic domain.
\end{enumerate}
\end{lemma}

\begin{proof}
\textit{Part 1.} Assume that the support
$A = \supp(m_s)$ is a hyperbolic set, and let $n$ be such that
$|(R^n)'(x)| \geq t > 1$ for all $x \in A$. Since $s < 0$,
\[
m_s(A) = \int_{R^{-n}(A)} |(R^n)'|^s \, dm_s
= \int_A |(R^n)'|^s \, dm_s
\leq t^s m_s(A) < m_s(A),
\]
which is a contradiction.

\smallskip
\textit{Part 2.} Let $U \subset F(R)$ be a topological disk around an
attracting or superattracting periodic point of period $p$, with
$R^p(U) \Subset U$ and $|(R^p)'(x)| \leq w < 1$ for every $x \in U$. We
claim that $m_s(U) = 0$. Indeed,
\[
m_s(U) = \int_{(R^p)^{-1}(U)} |(R^p)'|^s \, dm_s
\geq \int_U |(R^p)'|^s \, dm_s
\geq w^s m_s(U) > m_s(U),
\]
since $s < 0$ and $w < 1$ imply $w^s > 1$. Thus $m_s(U) = 0$, as
claimed. By the same argument,
$m_s((R^p)^{-1}(U)) = 0$, and hence
$m_s((R^n)^{-1}(U)) = 0$ for every $n \geq 1$ by induction.
\end{proof}
\begin{proof}[Proof of Theorem~\ref{th.Uno}]
If $R$ is hyperbolic, then the conclusion of the theorem follows from
Lemma~\ref{lm.hyperbolic}. Conversely,  $R$ is $J$-stable, since
it is not an accumulation point of parabolic rational maps and, hence, does not admit summable critical points. If $R$ does not admit a
pseudoconformal measure with $s < 0$, then $R$ does not admit
degenerating sequences of periodic cycles.  It follows that $J(R)$ has zero Lebesgue measure (see \cite{PrzyRohd}). Therefore, by Theorem~D in \cite{MSS}, $R$ is hyperbolic. (See also \cite{LevinAnalytic}, \cite{MakRuelle},
\cite{PrzyRohd}, and \cite{RivConn}.)
\end{proof}

\begin{proof}[Proof of Corollary~\ref{cor.Rhyper}]
The proof follows immediately from Theorem~\ref{th.Uno} and
Theorem~D in \cite{MSS}.
\end{proof}

\subsection{Deformation spaces of rational maps}

\begin{lemma}\label{lm.qstable}
Let $R$ be a rational map such that every critical point has infinite
orbit and there are no critical relations. Then $R$ is $q$-stable if and
only if
\[
\dim(\operatorname{Teich}(R^q)) = \#(V(R^q)) = q(2\deg(R) - 2).
\]
\end{lemma}

\begin{proof}
This follows from the definitions of the Teichm\"uller space $\operatorname{Teich}(R^q)$ and the Hurwitz class $H(R^q)$.
\end{proof}

We immediately  obtain the following corollary:

\begin{corollary}\label{cor.SeparationLemma}
A rational map $R$ is $q$-stable if and only if
\[
\#(V(R^q)) = q(2\deg(R) - 2),
\]
and for every critical value $v_i\in V(R^q)$ there exists an $R^q$-invariant
Beltrami differential $\mu_i$ such that, for every critical value $v_j\in V(R^q)$,
we have
\[
\int_{\C}\gamma_{v_j}(z)\mu_i(z)|d z|^2 = \delta_{ij},
\]
where $\delta_{ij}$ is the Kroenecker delta function.
\end{corollary}

\begin{proof}[Proof of Theorem~\ref{th.qstable}] We can assume that the Fatou set $F(R^q)$ only contains attracting periodic components. Then
\[
\operatorname{Teich}(R^q)=\operatorname{Teich}(S_q)\otimes Teich(J(R^q)),
\]
where $S_q$ is the Riemann surface associated to the action of $R^q$ on $F(R^q)$. The space $\operatorname{Teich}(J(R^q))$ is the Teichmuller space associated to the action of $R^q$ on the Julia set $J(R^q)$. Then

\[
\operatorname{dim}(\operatorname{Teich}(S_q))=\#\{V(R^q) \cap F(R)\}
\]
and, by assumption
\[
\operatorname{dim}(\operatorname{Teich(J(R^q))})=\#\{V( R^q)\cap J(R)\},
\]
with generators $\mu_i$, with $i=1,...,\operatorname{dim}(\operatorname{Teich}(J(R^q)))$ satisfying
\begin{itemize}
\item $R^q$ acts ergodically on $\operatorname{supp}(\mu_i)$ for every $i$; and
 \item The Beltrami operator $B_R$ acts as a permutation  of the differentials  $\mu_i$.
\end{itemize}

By Bers's embedding theorem (\cite{CMFixedETDS, GardLakic,Makarxiv2001, McMSull}), the
 tangent space $T(\operatorname{Teich}(R^q))$ can be represented by $R^q$-invariant
harmonic differentials supported on $F(R)$, together with the space
\[
 X_q=T(\operatorname{Teich}(J(R^q)))\]
 \[
 =\{ \textnormal{measurable }
R^q\textnormal{-invariant Beltrami differentials on } J(R)\}.
\]
 The Beltrami operator
$B_R$ acts on the space $T(\operatorname{Teich}(R^q))$ as a bijective linear automorphism of
order $q$ leaving $X_q$ and $T(\operatorname{Teich}(S_q))$ invariant. Assume $q = kl$. We first consider $X_q$. Since
\[
Fix(B_{R^k}) = \{\mu \in X_q : B_{R^k}(\mu) = \mu\}
\]
is isomorphic to $X_k$, define
\[
P_l = \frac{1}{l}\,(I + B_{R^k} + B_{R^k}^2 + \cdots + B_{R^k}^{l-1}).
\]
Then $P_l : X_q\to \operatorname{Fix}(B_{R^k})$ is a linear projection with
\[
\dim\bigl(P_l(X_q)\bigr)
\;\geq\; \frac{\dim(X_q)}{l}.
\]
Since
\[
\dim(\operatorname{Fix}(B_{R^k})) = \dim(X_k)
\leq \#(V(R^k)\cap J(R)),
\]
we obtain
\[
\dim(X_k) = k\cdot(\#(V(R)\cap J(R))).
\]

Now,  recall  for every $n>1$
\[
\operatorname{dim}(\operatorname{Teich}(S_n))=\#\{ V(R^n) \cap F(R)\}
\]
\[
= n \cdot(\#\{V(R)\cap F(R)\}).
\]
Then,
\[
\operatorname{dim}(T(\operatorname{Teich}(R^q)))=q\cdot(\#\{ V(R) \cap F(R)\}+\#\{V(R)\cap F(R)\})
\]
\[
= q(2\deg(R) - 2)
\]

Hence, $R$ is  $k$-stable by Lemma~\ref{lm.qstable}.

\smallskip
\textit{Part 2.} If $R$ is a hyperbolic structurally stable map, then
for every $q \geq 1$ we have
\[
\dim(\operatorname{Teich}(R^q)) = \#(V(R^q)) = q(2\deg(R) - 2).
\]
By Lemma~\ref{lm.qstable}, $R$ is $q$-stable.
\end{proof}

As an immediate corollary we have:

\begin{corollary}\label{cor.qstable} Let $R$ be a structurable stable rational rational map. If $R$ is $q$-stable for $q\ge 2$ and $\lambda$ is a $q$-root of unit then, $\lambda$ is an eigenvalue for the Beltrami operator with eigenspace $E_\lambda$ satisfying:
\begin{itemize}
 \item $\operatorname{dim}(E_\lambda)=2\deg(R) -2$, and
 \item $E_\lambda$ separates the space of generated by the functions $\gamma_{R(c_i)}$ with $c_i\in C(R)$.
\end{itemize}

\end{corollary}

\begin{proof}
 The corollary is a consequence of similar arguments of the proof above. Let $T=\lambda\cdot B_R$ then, as in the proof of  Theorem \ref{th.qstable}, $T$ acts on the space of $R^q$-invariant Beltrami differentials as a permutation on each orbit of the space. Hence, as above, the operator
 \[
P_\lambda = \frac{1}{q}\,(I +T + T^2 + \cdots + T^{q-1})
\]
is a linear projection onto $E_\lambda$, with $\operatorname{dim}(E_\lambda)\geq \operatorname{dim}(X^q)/q=2 \operatorname{deg}(R)-2$. To finish the proof, it is enough to show that a non-zero element in $E_\lambda$ defines a non-zero functional on the space generated by the functions $\gamma_{R(c_i)}$, with $c_i\in C(R)$. Otherwise, there exists a non-zero $\mu\in E_\lambda$ with
\[
\int \gamma_{R(c_i)}(z)\mu(z)\,|dz|^2=0
\]
for all $c_i\in C(R)$. Then, by Proposition \ref{pr.MakRuelle}
\[
\int \gamma_v(z) \mu(z)\ |dz|^2 =0
\]
for every $v\in V(R^q)$, which gives the intended contradiction.

\end{proof}

\subsection{Instability of rational maps}

 For the proof of Theorem~\ref{th.Dos}, we need the following fact, which is an application of Fatou's lemma, see also \cite{CMFixedETDS, MakRuelle}. For convenience we provide a proof.

\begin{lemma}\label{lm.forMainth}
Let $f(z) \in L^1(\C)$ and suppose $R^*(f) = \lambda f$ with
$|\lambda| = 1$. Then the differential $|f(z)|\,|dz|^2$ defines a finite
nonzero $R$-invariant measure.
\end{lemma}

\begin{proof}
Note that
\begin{align*}
\|f\|_{L^1}
&= \|\lambda f\|_{L^1}
 = \|R_*(f)\|_1 =\int |R_*(f)(z)|\,|dz|^2  \\
&\le \sum \int_\C |f(\zeta_i(z))|\,|\zeta_i'(z)|^2 \,|dz|^2=\int_\C |R_*|(|f|) |dz|^2\\
&\le \int |f(z)|\,|dz|^2
 = \|f\|_{L^1}.
\end{align*}
The sum is taken over the branches $\zeta_i$ of $R^{-1}$.
Hence all inequalities above are equalities. By Fatou’s lemma,
\[
|f(z)| = \sum_i
  |f(\zeta_i(z))|\,|\zeta_i'(z)|^2=|R_*|(|f|),
\]
almost everywhere.  This proves the claim.
\end{proof}

\begin{proof}[Proof of Theorem~\ref{th.Dos}] We may assume that $R$ is structurally stable.
Let $\rho$ be a unimodular measure with $(s ,t) = (-1,1)$ and eigenvalue
$\lambda$. First, suppose that
$\lambda = \exp(2\pi i \theta)$ with rational $\theta=\frac{p}{q}$. We claim that $R$ is $q$-unstable.    To prove this, we use
the arguments of \cite{MakRuelle} and proceed by contradiction assuming   $R$ is  $q$-stable. Define
\[
f(z) = \int \gamma_a(z)\, d\rho(a),
\]
the generalized Cauchy transform of the measure $\rho$. Then $f(z)$ is
a non-zero integrable function, holomorphic off the support of $\rho$. Indeed, by \cite{GardLakic} or \cite{KrushkalQCRiem},
there exists a constant $C$ such that Fubini's theorem gives
\[
\begin{aligned}
\int_{\C} |f(z)|\, |dz|^2
&\le \int d|\rho|
   \int_{\C} |\gamma_a(z)|\, |dz|^2 \\
&\le C \int |a \ln |a||\, d|\rho| < \infty.
\end{aligned}
\]
Thus, by Proposition \ref{pr.MakRuelle}
\begin{align*}
R_*(f(z))
&= \int R_*(\gamma_a(z))\, d\rho(a) \\
&= \int \frac{1}{R'(a)}\,\gamma_{R(a)}(z)\, d\rho(a)
   + \sum_{c_i} \frac{1}{R''(c_i)}\,\gamma_{R(c_i)}(z)
     \int \gamma_a(c_i)\, d\rho(a) \\
&= \lambda f(z) + \sum_{c_i} \frac{1}{R''(c_i)}\,
   \gamma_{R(c_i)}(z)\, f(c_i).
\end{align*}

By iterating this relation $q$ times, we obtain
\begin{align}\label{eq:aiterated}
R_*^q(f)
&= \lambda^q f(z)
   + \sum_{c_i} \frac{1}{R''(c_i)}\, f(c_i)
   \bigl[
      \lambda^{q-1}\gamma_{R(c_i)}(z)
      + \lambda^{q-2} R_*(\gamma_{R(c_i)}(z))
      + \cdots \notag\\
&\qquad\qquad\qquad
      +\, R_*^{q-1}(\gamma_{R(c_i)}(z))
   \bigr].
\end{align}

By  Corollary~\ref{cor.SeparationLemma}, Proposition \ref{pr.MakRuelle} and the $q$-stability of $R$,  taking
the $R^q$-invariant  Beltrami differential $\mu_i$ satisfying
\[
\int \gamma_{R^{q}(c_i)}(z)\mu_i(z) |dz|^2=1
\]
and
\[
 \int \gamma_{v}(z) \mu_i(z) |d z|^2=0
\]
for every $v\in  V(R^q)\setminus {R^{q}(c_i)}$,  we get $f(c_i)=0$. As $\lambda^q=1$,

\[
 R_*^q(f(z)) = f(z), \, a.e.
\]
 Hence, by Lemma \ref{lm.forMainth}, the differential $|f(z)|\,|dz|^2$
defines a finite nonzero $R^q$- invariant measure and the $SC(R)$ has Lebesgue positive measure.  By  Lemma \ref{lem.SC},  this is a  contradiction with  the weak dissipativity of $R$, since  $R$ is not a Latt\'es map.

Now consider the case $\lambda = \exp(2\pi i \theta)$ with
$\theta \notin \mathbb{Q}$.
Assume that $R$ is $q$-stable for every $q\geq 1$. Then
as above, we claim $f(c)=0$ for $c\in C(R)$. Otherwise,  if $c$ satisfies   $f(c)\neq 0$  choose  $\epsilon > 0$  small enough and $q_\epsilon$ such
that
\[
\bigl|\lambda - \nu_\epsilon \bigr| \leq \epsilon
< \min_{c \in C(R), f(c)\neq 0}
   \left|\frac{f(c)}{R''(c)}\right|,
\]
where $\nu_\epsilon$ is a suitable $q_\epsilon$-root of $1$.
Using the $q_\epsilon$-stability of $R$, Corollary \ref{cor.qstable} and integrating
$\gamma_{R(c)}(z)$  with a suitable element of $E_{\nu_\epsilon}$ we have a contradiction.

 Hence
\[
R_*(f) = \lambda f,
\]
almost everywhere
which contradicts the weak dissipativity of $R$ by Lemmas \ref{lem.SC} and \ref{lm.forMainth}, which finishes the proof.
\end{proof}

\subsection{Unimodular vector fields}

\begin{proof}[Proof of Proposition \ref{pr.principal}]
Let $\rho$ be a $-1$-pseudoconformal measure and let $\mu$ be a
nonzero invariant Beltrami differential that is non-singular with
respect to $\rho$. We can assume that $|\mu| = 1, \, \rho$-almost everywhere.
Let $\nu(z) = \exp\!\left(i \tfrac{\Arg(\mu)}{2}\right)$. If
\[
\phi(z) = \overline{\nu(z)} \, \nu(R(z)) \frac{|R'(z)|}{R'(z)},
\]
then $\phi^2(z) =1, \, \rho$-almost everywhere. Therefore,
$\phi$ is a measurable function taking values in $\{-1, 1\}$ on
$\supp(\rho)$. Thus, $\supp(\rho)$ may be decomposed into two measurable
sets:
\[
S_1 = \phi^{-1}(1), \qquad S_2 = \phi^{-1}(-1).
\]
 As $\nu$ is non-singular with respect to a
$-1$-pseudoconformal measure $\rho$, it follows that either $S_1$ or $S_2$ has positive measure. Assume, for instance, that $\rho(S_2)>0$, then we claim that  $\nu_2 = \nu|_{S_2}$ determines a unimodular vector field with eigenvalue $-1$ as follows, Let $U_1=S_2\cup R^{-1}(S_2)$ and
\[
v_1(y)=
\begin{cases}
-\nu_2(R(y))\frac{|R'(y)|}{R'(y)}, & \text{if } y \in R^{-1}(S_2),\\
\nu_2(y),  & \text{if } y\in S_2.
\end{cases}
\]

Hence, $v_1$ is well-defined and for $\rho$-almost every $z\in U_1$ satisfies
\begin{equation}
-1= \overline{v_1(z)} \, v_1(R(z)) \frac{|R'(z)|}{R'(z)}.
\tag{1}
\label{eq:vector}
\end{equation}

Now by induction we obtain a vector field $\tilde{v}$ satisfying the equation \ref{eq:vector} $\rho$-almost everywhere in $\bigcup_{n>0} R^{-n}(S_2)$. Since $R$ is conservative on $\supp(\rho)$ we have $\supp(\rho)=\bigcup_{n>0} R^{-n}(S_2)$, $\rho$-almost everywhere. After modifying $v$ on a set $\rho$-measure zero, we obtain  a unimodular vector field with eigenvalue $-1$, which proves the claim and hence the proposition. \end{proof}

\begin{proof}[Proof of Proposition \ref{pr.pseudoconformalerg}]
We combine the arguments of \cite{StrictErgodFurstenberg}
(for example, Theorem~3.1) and Proposition \ref{pr.principal} above.
Let $v$ be a unimodular vector field with eigenvalue $\lambda$ for $R$.
Then the function $\phi(x, z) = v(x) z$ is not constant with respect to
the first variable and satisfies the condition
\[
\phi(R_1) = \lambda \phi.
\]

Now, let $\phi$ be a solution of the equation $\phi(R_1) = \lambda \phi$
for some $\lambda$. We can assume that $|\phi| = 1$ almost everywhere,
and hence $\phi \in L_2(\rho \times m)$. Let
\[
\phi(x, z) = \sum_{i = -\infty}^{\infty} a_i(x) z^i
\]
be the Fourier decomposition of $\phi$ with respect to the second variable.
If $a_1(x) \neq 0$ $\rho$-almost everywhere, then $a_1$ defines
a unimodular vector field with eigenvalue $\lambda$.

We claim that there exist $\lambda$ and $\phi$ such that
\[
\phi(R_1) = \lambda \phi, \qquad a_1(x) \neq 0 \quad
\rho\text{-almost everywhere}.
\]
Otherwise, there exists $n \in \mathbb{Z}$, $n \neq 1$, such that
$|a_n(x)| = 1$, $a_n$ is not constant $\rho$-almost everywhere, and
\[
a_n(R(x)) \left( \frac{R'(x)}{|R'(x)|} \right)^n
= \lambda a_n(x), \qquad \rho\text{-almost everywhere}.
\]

Let $g(x) = \exp\!\left(i \tfrac{\Arg(a_n(x))}{n}\right)$.
Then the function
\[
\tau(x) = \overline{g(x)} \, g(R(x)) \frac{R'(x)}{|R'(x)|}
\]
satisfies $\tau^n(x) = \lambda$ $\rho$-almost everywhere.
Let $\lambda_i$ be an $n$-th root of $\lambda$, and define
\[
S_i = \{x : \tau(x) = \lambda_i\}.
\]
Then there exists $j$ such that $\rho(S_j) > 0$ and for $\rho$-almost every $x\in S_j$ we have
\[
\overline{g(x)} \, g(R(x)) \frac{R'(x)}{|R'(x)|} = \lambda_j.
\]
By similar arguments as in the proof of Propostion \ref{pr.principal}, we construct an unimodular vector field $v$ with eigenvalue $\lambda_j$ which  gives a contradiction finishing the claim and the proof of the proposition.
\end{proof}

\begin{proof}[Proof of Proposition \ref{pr.existunimodular}]
This proposition is an application of ergodic theory; see, for example,  \cite{Aaronson}. According to Theorem~8.3.2
and Corollary~8.3.4 of \cite{Aaronson}, the map $\tilde{R}$ admits
a finite invariant measure absolutely continuous with respect to
$\rho \times m$ if and only if the cocycle
\[
\phi(z) = \Arg(R'(z)) - \int \Arg(R') \, d\rho
= h(z) - h(R(z))
\]
is a coboundary with transition function $h(z)$. If
$g(z) = \exp(i h(z))$, then
\[
\lambda \, \frac{R'(z)}{|R'(z)|} = \frac{g(z)}{g(R(z))},
\]
where $\lambda = \exp(-i \int \Arg(R')\, d\rho)$. This completes the proof.
\end{proof}
\subsection{Proof of Theorem \ref{th.unimodular}}

In this subsection, all rational maps $R$ and the individual
postcritical sets $P_c(R)$ involved satisfy the conditions of
Theorem \ref{th.unimodular}.

\begin{theorem}\label{th.mera}
Let $R$ be a rational map with a critical point $c$ satisfying
the conditions of Theorem \ref{th.unimodular}. Then the following
statements are true.
\begin{enumerate}
 \item The individual postcritical set $P_c(R)$ admits an ergodic
 $-1$-pseudoconfor\-mal measure $\rho$.
 \item The measure $\rho$ admits an equivalent invariant
 absolutely continuous probability measure.
 \item There exists a metric on $P_c(R)$ equivalent to the Euclidean
 metric for which the restriction of $R$ to $P_c(R)$ is an isometry.
\end{enumerate}
\end{theorem}

\begin{proof}
Part~1 follows from Theorem \ref{th.Uno}.

Part~2. It is enough to show that, under the assumptions of the theorem,
there exists a finite invariant measure $\mu$ absolutely continuous
with respect to $\rho$. Since the operator
\[
T_R(\phi) = \frac{\phi \circ R}{|R'|}
\]
is a continuous endomorphism of $L_1(P_c(R), \rho)$ with $\|T_R\| \leq 1$,
we consider
\[
F_n(x) = \frac{1}{n} \sum_{i = 0}^{n-1}
T_R^i(\chi_{P_c(R)})(x)
= \frac{1}{n} \sum_{i = 0}^{n-1} \frac{1}{|(R^i)'(x)|}.
\]
Then, by Theorem~4.9 in \cite{Krengel}, we have $\|F_n\| \leq 1$, and
the sequence $F_n$ converges in measure to a function $F$ with
$\|F\| \leq 1$, which is $T_R$-invariant; that is, $F(R) = |R'| F$.
Therefore, the measure $d\rho = F \, d\rho$ is a nonzero finite invariant
measure absolutely continuous with respect to $\rho$. The hypothesis
\[\liminf_{n\rightarrow \infty}\frac{1}{n}\sum_{j=0}^{n-1}\frac{1}{|(R^j)'(x)|}>0
\] for  $\rho$-almost every $x$ implies $F>0$ almost everywhere. Consequently
\[
d\nu= F\, d\rho
\] is a nonzero finie invariant measure. After normalization, $\nu$ is a probability measure equivalent to $\rho$.

Part~3. By bounded recurrence, $R$ is surjective and equicontinuous on $P_c(R)$.

Define the metric
\[
d(x, y) = \sup_n \operatorname{dist}(R^n(x), R^n(y)).
\]
Equicontinuity gives finiteness and equivalence with the Euclidean metric. Therefore $R$ is a
surjective continuous non-expansive map on the metric space $(P_c(R), d)$. By a classical result
of Freudenthal and Hurewicz (see \cite{YuLyuDiss}), $R$ is an isometry with respect to the metric $d$.
\end{proof}

Next, we consider the dynamics of a rational map $R$ on the unit
tangent bundle over the individual postcritical set $P_c(R)$, represented
by
\[
T_1 : P_c(R) \times \mathbb{S}^1 \to P_c(R) \times \mathbb{S}^1,
\qquad
T_1(x, v) = \left(R(x), \tfrac{R'(x)}{|R'(x)|} v\right).
\]
We gather the relevant properties of $T_1$ in the following corollary.

\begin{corollary}\label{cor.T1}
\leavevmode
\begin{itemize}
 \item The map $T_1$ is a non-singular invertible transformation of
$P_c(R) \times \mathbb{S}^1$ equipped with the finite measure
 $\rho \times m$, where $m$ is the Haar measure on the circle
 $\mathbb{S}^1$, sharing the Radon--Nikodym derivative with $R$ with
 respect to $\rho$.
 \item The map $T_1$ admits a probability invariant measure equivalent
 to $\rho \times m$, and hence there are no weakly wandering sets
 of positive $\rho \times m$-measure.
\end{itemize}
\end{corollary}

\begin{proof}
The first item follows from Part~3 of Theorem \ref{th.mera}.
The second follows from Part~2 of Theorem \ref{th.mera}.
\end{proof}

Now, for the proof of Theorem \ref{th.unimodular}, we need to show that
$T_1$ is not weakly mixing. To do so, we use arguments similar to those
in Proposition \ref{pr.pseudoconformalerg}, by considering the group of
essential values of a multiplicative cocycle (see \cite{SchmidtRig}).

\begin{definition}
Let $(X, \nu)$ be a probability space, and let $T : X \to X$ be an
invertible non-singular transformation. Let
$\phi : X \to \mathbb{S}^1$ be a measurable cocycle. Then the \emph{collection of essential
values} of $\phi$ is the set
\[
E(\phi) = \left\{ a \in \mathbb{S}^1 :
\nu\bigl(A \cap T^{-n}(A) \cap \{ \|\phi_n(x) - a\| < \epsilon \}\bigr) > 0
\right\},
\]
for every measurable set $A \subset X$ with $\nu(A) > 0$ and every
$\epsilon > 0$, for some $n \in \mathbb{Z}$, where $\|\cdot\|$ denotes
the chordal metric on $\mathbb{S}^1$.
\end{definition}

By Lemma~3.3, Proposition~3.8, Theorem~3.9, Proposition~3.12, and
Corollary~5.4 of \cite{SchmidtRig}, we have the following:

\begin{proposition}\label{pr.Essential}
Let $T : (X, \nu) \to (X, \nu)$ be an invertible non-singular
transformation of a Lebesgue probability space, and let
$\phi : X \to \mathbb{S}^1$ be a measurable function. Then the following
statements are true.
\begin{enumerate}
 \item $E(\phi)$ is a closed subgroup of $\mathbb{S}^1$.
 \item The skew-product $T_\phi : X \times \mathbb{S}^1 \to
 X \times \mathbb{S}^1$, equipped with the measure $\nu \times m$,
 is ergodic if and only if $T : X \to X$ is ergodic and
 $E(\phi) = \mathbb{S}^1$.
 \item If $E(\phi) \neq \mathbb{S}^1$, then there exists a
 multiplicative cocycle $\tilde{\phi} : X \to \mathbb{S}^1$
 cohomologous to $\phi$ with values in $E(\phi)$. In particular,
 $\phi$ is a coboundary if and only if $E(\phi) = \{1\}$, the unit
 of the group $\mathbb{S}^1$.
\end{enumerate}
\end{proposition}

For our purposes, let us first assume that
$\phi(x) = \tfrac{R'(x)}{|R'(x)|}$ and that $E(\phi) \neq \mathbb{S}^1$.
Then $P_c(R)$ admits a unimodular vector field non-singular with respect
to $\rho$. Indeed, by Part~2 of Proposition \ref{pr.Essential}, there
exists a measurable map $h : P_c(R) \to \mathbb{S}^1$ such that
\[
s(x) = \frac{R'(x)}{|R'(x)|} \cdot \frac{h(R(x))}{h(x)} \in E(\phi),
\]
$\rho$-almost everywhere. But $E(\phi)$ is a finite subgroup of
$\mathbb{S}^1$. If $\lambda_i \in E(\phi)$ and
$A_i = \{x : s(x) = \lambda_i\}$, then for some $i_0$ we have
$\rho(A_{i_0}) > 0$. Then $h$ restricted to $A_{i_0}$ determines
a unimodular vector field with eigenvalue $\lambda_{i_0}$ as in  Proposition \ref{pr.principal}.

Next, we may assume that $R$ is ergodic on $P_c(R)$ with respect to
$\rho$. By Part~2 of Proposition \ref{pr.Essential}, the equality
$E(\phi) = \mathbb{S}^1$ is equivalent to the ergodicity of $T_1$ on
$P_c(R) \times \mathbb{S}^1$ with respect to $\rho \times m$. To proceed,
we use the following result from ergodic theory (\cite{DistalParry}).

\begin{definition}
Let $T : X \to X$ be a measure-preserving transformation of a compact
space $X$ equipped with a probability measure $\nu$. Let
\[
X = S_0 \supset S_1 \supset \cdots
\]
be a decreasing sequence of measurable sets $S_i$ of positive measure
such that $\displaystyle{\lim_{n \to \infty} \nu(S_n)} = 0$. The sequence $(S_n)$ is
called a \emph{separating sieve} for $T$ if there exists a subset
$X_0 \subset X$ of full measure with the property that, given
$x, y \in X_0$, if for each $N$ there exists $n$ such that
$T^n(x), T^n(y) \in S_N$, then $x = y$.
\end{definition}

By Theorem~3 of \cite{DistalParry}, we have the following proposition.

\begin{proposition}\label{pr.Parry}
An ergodic transformation $T$ on a Lebesgue space that admits a separating sieve possesses a non-constant eigenfunction (that is, $T$ is not
weakly mixing). In particular, $T$ has zero entropy.
\end{proposition}

We claim that $T_1$ is not weakly mixing (i.e., it has non-constant
eigenfunctions). Indeed, let
\[
U_n = \{z: |z-R(c)|\leq 2^{-n}\}, \qquad
V_n = \{ v \in \mathbb{S}^1 : \|v - 1\| < 2^{-n} \}, \qquad
S_n = U_n \times V_n,
\]
then $\{S_n\}$ defines a separating sieve for $T_1$ on
$P_c(R) \times \mathbb{S}^1$ by Part~3 of Theorem \ref{th.mera}.
Thus $T_1$ is an ergodic invertible measure-preserving transformation
by Part~2 of Theorem \ref{th.mera}. By Proposition \ref{pr.Parry},
the claim follows.

To finish the proof of Theorem \ref{th.unimodular}, let $f$ be a
non-constant eigenfunction for $T_1$ with eigenvalue $\lambda \neq 1$.
Then
\[
f(x, v) = \sum_{i = -\infty}^{\infty} a_i(x) v^i.
\]
Hence, $a_1(x)$ determines a unimodular vector field whenever
$a_1(x) \not\equiv 0$ $\sigma$-almost everywhere. Otherwise, there exists
$N$ such that $a_N(x) \not\equiv 0$ and
\[
a_N(R(x)) \left( \frac{R'(x)}{|R'(x)|} \right)^N
= \lambda a_N(x), \qquad \rho\text{-almost everywhere}.
\]

Since $\rho$ is ergodic, $|a_N(x)|$ is constant, so we can assume
$a_N(x) = \exp(i \phi(x))$. We may also assume that
$\left(\tfrac{R'(x)}{|R'(x)|}\right)^N$ is not constant
$\rho$-almost everywhere. Indeed, if
$\left(\tfrac{R'(x)}{|R'(x)|}\right)^N$ is constant, take
$\gamma \in \operatorname{Mob}$ such that $\gamma'$ is not constant and
$\gamma(P_c(R)) \subset \mathbb{C}$. Then
$R_\gamma = \gamma \circ R \circ \gamma^{-1}$ satisfies the conditions
of Theorem \ref{th.unimodular}, and a straightforward computation gives
\[
\frac{\gamma' \circ \gamma^{-1}(R_\gamma)}{\gamma' \circ \gamma^{-1}}
R'(\gamma^{-1}) = R'_\gamma.
\]
If
\[
g(x) = \frac{\gamma' \circ \gamma^{-1}(x)}{|\gamma' \circ \gamma^{-1}(x)|},
\qquad
\frac{R'}{|R'|} = k \quad \text{on } P_c(R), \ \rho\text{-almost everywhere},
\]
then
\[
\frac{g(R(x))}{g(x)} \cdot k
= \left( \frac{R'_\gamma(x)}{|R'_\gamma(x)|} \right)^N,
\]
a contradiction.

Now, let $\psi(x) = \exp\!\left( \tfrac{i \phi(x)}{N} \right)$. Then
\[
\tau(x) = \overline{\psi(x)} \, \psi(R(x)) \frac{R'(x)}{|R'(x)|}
\]
satisfies $\tau^N(x) \equiv \lambda$, $\rho$-almost everywhere.
Let $j$ be an $N$th root of $\lambda$, and set
\[
S_j = \{ x \in P_c(R) : \tau(x) = j \}.
\]
If $\rho(S_j) > 0$, then $\psi$ determines a unimodular vector field
with eigenvalue $j$ as in Proposition \ref{pr.principal}. Hence, the measure $\psi(\rho)$  determines a nonzero unimodular measure.   This completes the proof of Theorem \ref{th.unimodular}.

Combining Theorem \ref{th.Schmidt} with Atkinson’s theorem
\cite{Atkinson1976}, we obtain the existence of an automorphic measure.  We regard the eigenvalue of $T_1$ as a
“complex Lyapunov exponent.”

\begin{proof}[Proof of Theorem \ref{th.conservative}]
Take $X = \operatorname{supp}(\rho)$ and
$f = (f_1, g) \colon X \to \mathbb{R}^2$ to be the cocycle with
 $f_1 = -\ln |R'|$. By assumption and Theorem~2 of
\cite{SchmidtRecurrence}, $f$ is a conservative cocycle. According
to Theorem \ref{th.Schmidt}, for every positive-measure set
$B \subset X$, there exists a non-atomic $\sigma$-finite invariant
ergodic measure $\nu_B$ such that $\nu_B(B) > 0$ and
$f$ is a coboundary for $R$ with respect to $\nu_B$.

Hence, there exists a transition function
$\psi = (\psi_1, \psi_2) \colon X \to \mathbb{R}$ with
\[
f_1 = \psi_1(R) - \psi_1,
\qquad
g = \psi_2(R) - \psi_2,
\]
$\nu_B$-almost everywhere. Then the measure given by
\[
dm = \exp(\psi_1 +i \psi_2) \, d\nu_B
\]
is the desired $(-1,1)$-automorphic measure.
\end{proof}
\begin{proof}[Proof of  Proposition \ref{pr.pseudoconformal}]
The  proposition follows from Theorem \ref{th.conservative},
Proposition \ref{pr.absolutelycontinuous}.1, and Proposition \ref{pr.Atkinson}.
\end{proof}

\begin{proof}[Proof of Theorem \ref{th.boundedvel}]
As in the proof of Theorem \ref{th.unimodular}, this proof uses ergodic
arguments. Let $R$ be a rational map and let $v \in \overline{\C}$ be a point
of bounded velocity with respect to $R$, with $0 \leq \rho \leq 1$. Then
\[
\sum_{i = 0}^{n - 1} \Arg\!\bigl(R'(R^i(v))\bigr)
= \Arg\!\bigl((R^n)'(v)\bigr) + 2\pi \sum_{i = 1}^{n - 1} K_i
\]
\[
 = \Arg\!\bigl((R^n)'(v)\bigr)+2\pi\rho n+ O(1).
\]
Let
\[
g(z) = \Arg(R'(z)) - 2\pi\rho
\]
be a real-valued additive cocycle. For every $n$, the function $g(n, z)$
is bounded.
By assumption, $g(n, v) = O(1)$ for $n \geq 0$.

Let $M = \sup_n |g(n, v)|$. We claim that
\[
|g(n, x)| \leq 2M
\]
for every $x \in JC(\Omega_R(v)) $. Indeed, by the equality
\[
g(n + m, x) = g(n, x) + g(m, R^n(x)),
\]
we conclude that $|g(m, R^n(v))| \leq 2M$ for every $m, n \in \mathbb{N}$.
Fix $x \in JC(\Omega_R(v))$, $m \in \mathbb{N}$, and a sequence
$\{n_i\}$ such that $R^{n_i}(v) \to x$. Since $g(m, x)$ is continuous at $x$,
we obtain $|g(m, x)| \leq 2M$, as claimed.

Let $h(x) = \limsup_n g(n, x)$. As $JC(\Omega_R(v))$ is $R$-invariant, then $h(x)$ is a bounded measurable transition
function for $g$ on $JC(\Omega_R(v))$, see also Proposition \ref{pr.absolutelycontinuous}.2. The function $\exp(i h(x))$ defines a unimodular vector
field with eigenvalue $\exp(2\pi i \rho)$, which completes the proof.
\end{proof}

\begin{proof}[Proof of Corollary \ref{cor.doubling}]
Arguing by contradiction, assume that $R$ is a structurally stable rational map.
If  $c$ is a critical point satisfying the assumptions, then $R'$ restricted
to $P_c(R)$ is real, and hence $P_c(R)$ has zero Lebesgue measure. Moreover,
the series
\[
\sum_{n = 1}^\infty \frac{1}{|R'(R^n(c))|}
\]
diverges by \cite{MakRuelle}, and $\chi_-(R(c)) = 0$, then $P_c(R)$
admits a $(-1)$-pseudoconformal measure $\rho$ by
Theorem~\ref{tm.pseudoconformal}.

To obtain a contradiction, we first claim that there exists a unimodular vector field $w$ with eigenvalue
$\lambda = 1$, non-singular with respect to $\rho$.

By assumption and the result of Section 1.2  of  ~\cite{Isola}, the critical value $v = R(c)$
has bounded velocity, and we may take $\rho = \tfrac{1}{3}$. By
Theorem~\ref{th.boundedvel}, there exists a $(-1,1)$-unimodular
vector field $w$ with eigenvalue $\lambda = \exp(-2\pi i/3)$; that is,
\[
w(R(z)) \frac{R'(z)}{|R'(z)|} = \exp(2 \pi i/3)\, w(z),
\]
on $JC(\Omega_R(R(c)))$. According to the remark after Corollary \ref{cor.doubling}, $JC(\Omega_R(R(c)))=\Omega_R(R(c))\setminus \Crit(R)$.
So the previous equality holds $\rho$-almost everywhere since $\rho$ is non-atomic and $Crit(R)$ is countable.

Let $u(z) = w(z)^3$. Since
\[
\left( \frac{R'(z)}{|R'(z)|} \right)^3 = \frac{R'(z)}{|R'(z)|},
\qquad \rho\text{-almost everywhere},
\]
we conclude that
\[
u(R(z)) \frac{R'(z)}{|R'(z)|} = u(z),
\]
which proves the claim.

Then $d\nu = u\, d\rho$ is a unimodular measure with eigenvalue $\lambda=1$.

Now, let
\[
F(z)=\int \gamma_a(z)\, d\nu(a)
\]
be its generalized Cauchy transform. By Fubini's theorem, $F\not \equiv 0$ is absolutely integrable over $\overline{\mathbb{C}}$. Since
\[
\operatorname{supp}(\nu)=\operatorname{supp}(\rho)
\]
has zero Lebesgue measure,  Corollary 8.3 of \cite{GamelinUniform} implies that outside $\operatorname{supp}(\rho)$ the function $F(z)$ is holomorphic and non identically zero.

By assumption, an argument similar to that used in the proof of  Theorem~\ref{th.Dos} gives
\[
R_*(F(z)) = F(z).
\]
Thus, $F$ is a fixed point of the Ruelle operator.
 Hence,  by Corollary~12 in \cite{MakRuelle}, the function
\[
\phi(z)=\frac{|F(z)|}{F(z)}
\]
satisfies
\[
\phi(R(z))\frac{\overline{R'(z)}}{R'(z)}=\phi(z)
\]
for Lebesgue almost every $z$ with $\phi(z)\neq 0$.

In particular, $\phi$ determines a holomorphic line field in the sense of McMullen which is non-zero on an open set. By Lemma~3.16 of \cite{McMullenbook}, $R$ is a flexible Latt\'es map, which contradicts the structural stability of $R$. This completes the proof.
\end{proof}

 \bibliographystyle{amsplain}
\bibliography{workbib}
%\Addresses\
\end{document}